\documentclass[final,onefignum,onetabnum,onealgnum]{siamart251216}

\usepackage{amsmath,amsfonts,amsopn,amssymb}
\usepackage{graphicx}
\usepackage[caption=false]{subfig}
\usepackage{multirow,relsize,makecell}
\usepackage{url}
\ifpdf
  \DeclareGraphicsExtensions{.eps,.pdf,.png,.jpg}
\else
  \DeclareGraphicsExtensions{.eps}
\fi
\usepackage[shortlabels]{enumitem}
\usepackage{accents}

\usepackage{algorithm,algorithmic}

\numberwithin{theorem}{section}
\newsiamremark{remark}{Remark}
\newsiamremark{assumption}{Assumption}
\newsiamremark{example}{Example}
\crefname{remark}{Remark}{Remarks}
\crefname{assumption}{Assumption}{Assumptions}
\crefname{example}{Example}{Examples}

\headers{Randomized subspace correction methods}{Boou Jiang, Jongho Park, and Jinchao Xu}

\title{Randomized subspace correction methods\\ for convex optimization\thanks{Submitted to arXiv.\funding{This work was supported by the KAUST Baseline Research Fund.}}
}

\author{
Boou Jiang\thanks{Applied Mathematics and Computational Sciences Program, Computer, Electrical and Mathematical Science and Engineering Division, King Abdullah University of Science and Technology~(KAUST), Thuwal 23955, Saudi Arabia
 (\email{boou.jiang@kaust.edu.sa},
 \email{jongho.park@kaust.edu.sa}, 
 \email{jinchao.xu@kaust.edu.sa}).
}
\and
Jongho Park\footnotemark[2]
\and
Jinchao Xu\footnotemark[2]
 }

\ifpdf
\hypersetup{
  pdftitle={Randomized subspace correction methods for convex optimization},
  pdfauthor={Boou Jiang, Jongho Park, and Jinchao Xu}
}
\fi

\begin{document}

\maketitle

\begin{abstract}
This paper introduces an abstract framework for randomized subspace correction methods for convex optimization, which unifies and generalizes a broad class of existing algorithms, including domain decomposition, multigrid, and block coordinate descent methods.
We provide a convergence rate analysis ranging from minimal assumptions to more practical settings, such as sharpness and strong convexity.
While most existing studies on block coordinate descent methods focus on nonoverlapping decompositions and smooth or strongly convex problems, our framework extends to more general settings involving arbitrary space decompositions, inexact local solvers, and problems with weaker smoothness or convexity assumptions.
The proposed framework is broadly applicable to convex optimization problems arising in areas such as nonlinear partial differential equations, imaging, and data science.
\end{abstract}



\section{Introduction}
\label{Sec:Introduction}
The main purpose of this paper is to develop an abstract framework for randomized subspace correction methods in convex optimization.
While subspace correction methods~\cite{TX:2002,Xu:1992} generalize a broad class of iterative algorithms, convex optimization itself encompasses a wide range of applications, including nonlinear partial differential equations~(PDEs), imaging, and data science.
Consequently, the proposed framework is flexible and applicable to a very broad range of problems.
Moreover, it accommodates highly general settings, such as weaker smoothness and convexity assumptions~\cite{Park:2022b,RD:2020}.

Subspace correction methods~\cite{Xu:1992} follow a divide-and-conquer strategy by decomposing the original problem into local subproblems defined on subspaces, which are solved independently. Many classical and modern iterative methods, including block relaxation, domain decomposition, and multigrid methods, can be viewed as instances of subspace correction methods. The theory has evolved over the past decades, covering both linear~\cite{Xu:1992,XZ:2002} and nonlinear problems~\cite{Carstensen:1997,Park:2020,TX:2002}.

Block coordinate descent methods are prominent examples of subspace correction methods for convex optimization. These methods solve local subproblems restricted to blocks of coordinates, often via gradient or proximal steps. More general updates, like upper bound minimization~\cite{HWRL:2017,RHL:2013}, are also possible. Their computational efficiency has led to widespread adoption. Key early results include~\cite{Tseng:2001,XY:2013}, and a comprehensive survey appears in~\cite{Wright:2015}. Recent advances cover convergence of cyclic~\cite{BT:2013,SH:2015}, randomized~\cite{Nesterov:2012,RT:2014}, accelerated~\cite{LLX:2015,LX:2015}, and parallel variants~\cite{FR:2015,NC:2016,RT:2016}. Applications include deep neural network training~\cite{ZLLY:2019,ZB:2017}.

In numerical analysis, domain decomposition and multigrid methods are essential examples of subspace correction methods. Their convergence for smooth convex optimization was studied in~\cite{TE:1998,TX:2002}, and later extended to constrained and nonsmooth cases~\cite{Badea:2006,BK:2012,BTW:2003,Park:2020}. These methods have been applied to various nonlinear variational problems, including PDEs~\cite{CHW:2020a,LP:2025a,Park:2024a}, variational inequalities~\cite{BK:2012,BTW:2003,Park:2024c}, elastoplasticity~\cite{Carstensen:1997}, and mathematical imaging~\cite{CTWY:2015,HL:2015,LP:2019}.

Subspace correction methods are classified as parallel or successive depending on the order of subproblem updates~\cite{TW:2005,Xu:1992}. In parallel methods (additive Schwarz), all subproblems are solved concurrently; in successive methods (multiplicative Schwarz), subproblems are solved sequentially.

Randomizing the order of subproblem updates leads to randomized subspace correction methods~\cite{GO:2012,HXZ:2019}, which often exhibit better performance compared to fixed-order approaches.
For quadratic optimization or linear problems, improved worst-case convergence rates under randomization are established in~\cite{HXZ:2019}.
A notable result in~\cite{SY:2021} shows that the worst-case complexities of cyclic coordinate descent methods and randomized coordinate descent methods are $\mathcal{O} (J^{4} \kappa \log(1/\epsilon) )$ and $\mathcal{O} (J^{2} \kappa \log(1/\epsilon) )$, respectively, where $J$ denotes the dimension of the variable, $\epsilon$ is a prescribed tolerance for the energy error, and $\kappa$ is a problem-dependent factor.
More precisely, by constructing a specific quadratic objective, it was shown that both complexity bounds can be attained simultaneously.
This demonstrates that cyclic coordinate descent methods can be up to $\mathcal{O}(J^{2})$ times worse than randomized coordinate descent methods, thereby clarifying the role of randomization: when the optimal ordering of subspaces is unknown, randomization can yield better performance than certain fixed-order strategies.
One may also refer to~\cite{CP:2016} for an analysis of how randomization averages antisymmetric terms in descent inequalities. This mechanism is beneficial for the development of accelerated methods~\cite{LLX:2015,LX:2015} and has motivated extensive research on randomized methods in convex optimization~\cite{CHW:2020b,NC:2016,Nesterov:2013,RT:2014}.
In particular,~\cite{CHW:2020b} studies randomized subspace correction methods for convex optimization with linear local problems, which can be viewed as a preliminary work for the present paper.

This paper introduces an abstract framework for randomized subspace correction methods for convex composite optimization~\cite{Nesterov:2013} on reflexive Banach spaces, accommodating diverse levels of smoothness and convexity~\cite{Park:2022b,RD:2020}.
The framework unifies a wide range of decomposition strategies, including block partitioning~\cite{BT:2013,Nesterov:2012,RT:2014} and overlapping domain decompositions commonly used in the numerical solution of PDEs~\cite{LP:2025a,Park:2024a,TW:2005,Xu:1992}. It supports both exact and inexact local solvers, encompassing methods such as coordinate descent, Bregman descent~\cite{DL:2015,GLLC:2020,HY:2016}, and constraint decomposition~\cite{BF:2024,Tai:2003}.
We establish convergence theorems that extend recent results, including those in~\cite{GO:2012,RT:2014}.
More precisely, by identifying a structural relationship between randomized and parallel subspace correction methods (see~\cite{HXZ:2019} for the linear case), we analyze the convergence of randomized subspace correction methods by leveraging existing analyses of parallel methods, such as~\cite{LP:2025b,Park:2020,PX:2024}. Furthermore, we provide new analyses that yield sharper estimates under stronger, yet commonly adopted, assumptions; see, for example,~\cite{RT:2014}.

In \cref{Table:theorems}, we summarize the main results of this paper.

\begin{table}
    \centering
    \resizebox{\textwidth}{!}{
    \begin{tabular}{c|ccc}
        \textbf{Results} & \textbf{Required assumptions} & \textbf{Convergence rates} & \textbf{Key constants} \\
        \hline
        \cref{Thm:conv} & \makecell{local problems (\cref{Ass:local}) \\ stable decomposition (\cref{Ass:stable})} & sublinear & \makecell{$q$, $\omega$, $\rho$ (\cref{Ass:local}) \\ $C_{K_0}$ (\cref{Ass:stable})} \\
        \hline
        \cref{Thm:conv_sharp} & \makecell{local problems (\cref{Ass:local}) \\ stable decomposition (\cref{Ass:stable}) \\ sharpness (\cref{Ass:sharp})} & \makecell{linear ($p = q$) \\ sublinear ($p > q$)} & \makecell{$q$, $\omega$, $\rho$ (\cref{Ass:local}) \\ $C_{K_0}$ (\cref{Ass:stable}) \\ $p$, $\mu_{K_0}$ (\cref{Ass:sharp})} \\
        \hline
        \cref{Rem:conv_sharp} & \makecell{local problems (\cref{Ass:local}) \\ global stable decomposition (\cref{Rem:Psi}) \\ sharpness (\cref{Ass:sharp})} & linear ($p = q$) & \makecell{$q$, $\omega$, $\rho$ (\cref{Ass:local}) \\ $C_{V}$ (\cref{Rem:Psi}) \\ $p$, $\mu_{K_0}$ (\cref{Ass:sharp})}\\
        \hline
        \cref{Thm:conv_strong} & \makecell{local problems (\cref{Ass:local}) \\ stable decomposition (\cref{Ass:stable}) \\ strong convexity (\cref{Ass:strong})} & linear ($q = 2$) & \makecell{$q$, $\omega$, $\rho$ (\cref{Ass:local}) \\ $C_{K_0}$ (\cref{Ass:stable}) \\ $\mu_{K_0}$, $\mu_{F, K_0}$ (\cref{Ass:strong})} \\
        \hline
        \cref{Rem:conv_strong} & \makecell{local problems (\cref{Ass:local}) \\ global stable decomposition (\cref{Rem:Psi}) \\ strong convexity (\cref{Ass:strong})} & linear ($q = 2$) & \makecell{$q$, $\omega$, $\rho$ (\cref{Ass:local}) \\ $C_V$ (\cref{Rem:Psi}) \\ $\mu_{K_0}$, $\mu_{F, K_0}$ (\cref{Ass:strong})}
    \end{tabular}
    }
    \caption{Main convergence theorems of this paper, and corresponding required assumptions, convergence rates, and key constants.
    All rates are for the expected energy error.}
    \label{Table:theorems}
\end{table}

The remainder of this paper is organized as follows.
In \cref{Sec:MSC}, we present an abstract framework of randomized subspace correction methods for convex optimization.
In \cref{Sec:Convergence}, we derive convergence theorems under various conditions on the target problem.
In \cref{Sec:Related}, we explain how the proposed framework relates to existing results.
In \cref{Sec:Applications}, we summarize possible applications of the proposed framework from diverse fields of science and engineering.
In \cref{Sec:Numerical}, we provide numerical results for randomized subspace correction methods.
Finally, in \cref{Sec:Conclusion}, we conclude the paper with some remarks.

\section{Subspace correction methods}
\label{Sec:MSC}
This section presents an abstract framework for randomized subspace correction methods for convex optimization.
In particular, we show that the convergence analysis of randomized subspace correction methods for convex optimization can be carried out within the framework of parallel subspace correction methods~\cite{Park:2020}, extending the analogy previously established for linear problems~\cite{GO:2012,HXZ:2019}. 
The proposed framework is highly versatile, accommodating diverse space decomposition settings for the model problem, a broad range of smoothness and convexity levels in the objective functional, and various types of inexact local solvers.

Let \( V \) be a reflexive Banach space equipped with the norm \( \| \cdot \| \).
Its topological dual is denoted by \( V^* \), and the duality pairing between \( V^* \) and \( V \) is written as
\[
\langle p, v \rangle = p(v), \quad p \in V^*,\ v \in V.
\]

Throughout this paper, we adopt the convention $0/0 = 0$ for arguments of $\sup$ and $0/0 = \infty$ for arguments of $\inf$.

\subsection{Space decomposition and subspace correction}
We consider the following abstract convex optimization problem:
\begin{equation}
\label{model}
    \min_{v \in V} \left\{ E(v) := F(v) + G(v) \right\},
\end{equation}
where \( F \colon V \to \mathbb{R} \) is a G\^{a}teaux differentiable and convex functional, and \( G \colon V \to \overline{\mathbb{R}} \) is a proper, convex, and lower semicontinuous functional.
The problem~\eqref{model} is referred to as a composite optimization problem~\cite{Nesterov:2013}, as it involves a nonsmooth term $G$ in addition to the smooth term $F$.
We further assume that the energy functional \( E \) is coercive, which guarantees the existence of a minimizer \( u \in V \) for the problem~\eqref{model}.

We assume that the solution space $V$ of~\eqref{model} admits a space decomposition of the form
\begin{equation}
\label{space_decomposition}
    V = \sum_{j=1}^J V_j,
\end{equation}
where each $V_j$, $j \in [J] = \{ 1, 2, \dots, J \}$, is a closed subspace of $V$.
The space decomposition~\eqref{space_decomposition} covers various algorithms, including block coordinate descent methods~\cite{Nesterov:2012,RT:2014}, domain decomposition methods~\cite{Park:2020,TW:2005}, and multigrid methods~\cite{TX:2002,XZ:2017}.
It is well known~\cite[Equation~(2.15)]{XZ:2002} that the space decomposition~\eqref{space_decomposition} satisfies the stable decomposition property.
Namely, for any \( q \in [1, \infty) \), we have
\begin{equation}
\label{norm_stable}
\sup_{\| w \| = 1} \inf_{\sum_{j=1}^J w_j = w} \left( \sum_{j=1}^J \| w_j \|^q \right)^{\frac{1}{q}} < \infty,
\end{equation}
where \( w \in V \) and \( w_j \in V_j \).


Subspace correction methods involve solving local problems defined on subspaces \( \{ V_j \}_{j=1}^J \).
For a given \( v \in V \), the optimal residual in a subspace \( V_j \) is obtained by solving the local minimization problem
\begin{equation}
\label{local_exact}
\min_{w_j \in V_j} E(v + w_j).
\end{equation}
Alternating minimization methods~\cite{Beck:2015,BT:2013} and certain domain decomposition methods (see, e.g.,~\cite{BK:2012,LP:2025a,Park:2024a,TX:2002}) fall into the category of subspace correction methods with exact local solvers as in~\eqref{local_exact}.
In contrast, block coordinate descent methods typically solve the local problem~\eqref{local_exact} inexactly, often using a single iteration of gradient descent~\cite{BT:2013,Nesterov:2012}, proximal descent~\cite{LX:2015,RT:2014}, or Bregman descent~\cite{DL:2015,GLLC:2020,HY:2016}.
Some methods further employ surrogate techniques, where~\eqref{local_exact} is replaced by an approximate problem with lower computational complexity; see, e.g.,~\cite{BF:2024,CHW:2020a,Tai:2003}.

To encompass all these methods, following~\cite{Park:2020,PX:2024}, we consider local problems of the form
\begin{equation}
\label{local_inexact}
    \min_{w_j \in V_j} \left\{ E_j(w_j; v) := F_j(w_j; v) + G_j(w_j; v) \right\},
\end{equation}
where \( F_j(\cdot ; v) \colon V_j \to \mathbb{R} \) and \( G_j(\cdot ; v) \colon V_j \to \overline{\mathbb{R}} \) are convex functionals for each \( v \in V \).
The functionals \( F_j(\cdot ; v) \) and \( G_j(\cdot ; v) \) serve as approximations to the exact local functionals \( F(v + \cdot) \) and \( G(v + \cdot) \) on $V$, respectively.
An example of~\eqref{local_inexact} corresponding to a proximal descent step is presented in \cref{Ex:local_inexact}.
Additional examples can be found in~\cite[Section~6.4]{Park:2020}.

\begin{example}
\label{Ex:local_inexact}
If we set
\begin{equation*}
F_j(w_j; v) = F(v) + \langle F'(v), w_j \rangle + \frac{1}{2\tau_j} \| w_j \|^2,\
G_j (w_j; v) = G (v + w_j), \quad
v \in V,\ w_j \in V_j,
\end{equation*}
for some \( \tau_j > 0 \), then the local problem~\eqref{local_inexact} corresponds to a single proximal descent step~\cite{LX:2015,RT:2014} with step size \( \tau_j \) for minimizing \( E (v + w_j) \).
\end{example}

The abstract parallel subspace correction method for solving the convex optimization problem~\eqref{model}, based on the space decomposition~\eqref{space_decomposition} and local solvers~\eqref{local_inexact}, is presented in \cref{Alg:PSC}.

\begin{algorithm}
\caption{Parallel subspace correction method for~\eqref{model}}
\begin{algorithmic}[]
\label{Alg:PSC}
\STATE Given the step size $\tau > 0$:
\STATE Choose $u^{(0)} \in \operatorname{dom} G$.
\FOR{$n=0,1,2,\dots$}
    \FOR{$j \in [J]$ \textbf{in parallel}}
        \STATE $\displaystyle
        w_j^{(n+1)} \in \operatornamewithlimits{\arg\min}_{w_j \in V_j} E_j (w_j; u^{(n)})
        $
    \ENDFOR
    \STATE $\displaystyle
    u^{(n+1)} = u^{(n)} + \tau \sum_{j=1}^J w_j^{(n+1)}
    $
\ENDFOR
\end{algorithmic}
\end{algorithm}

Another type of subspace correction method is the successive subspace correction method, in which the local problems in the subspaces are solved sequentially.  
In this paper, we focus on a particular variant known as the randomized subspace correction method, where the order of the local problems is chosen randomly; see \cref{Alg:RSC}.

\begin{algorithm}
\caption{Randomized subspace correction method for~\eqref{model}}
\begin{algorithmic}[]
\label{Alg:RSC}
\STATE Choose $u^{(0)} \in \operatorname{dom} G$.
\FOR{$n=0,1,2,\dots$}
    \STATE Sample $j \in [J]$ from the uniform distribution on $[J]$.
    \STATE $\displaystyle
        w_j^{(n+1)} \in \operatornamewithlimits{\arg\min}_{w_j \in V_j} E_j (w_j; u^{(n)})
        $
    \STATE $\displaystyle
    u^{(n+1)} = u^{(n)} + w_j^{(n+1)}
    $
\ENDFOR
\end{algorithmic}
\end{algorithm}

\subsection{Descent property}
In what follows, we denote by \( d \) and \( d_j \) the Bregman divergences associated with \( F \) and \( F_j \), respectively:
\begin{align*}
d(w; v) = F(v + w) - F(v) - \langle F'(v), w \rangle, \quad & v, w \in V, \\
d_j(w_j; v) = F_j(w_j; v) - F_j(0; v) - \langle F_j'(0; v), w_j \rangle, \quad & v \in V,\ w_j \in V_j.
\end{align*}

To ensure the convergence of the randomized subspace correction method, we adopt the assumptions on the local problem~\eqref{local_inexact} summarized in \cref{Ass:local}.
We note that \cref{Ass:local} provides a more general framework than several recent works, as it extends the smooth settings in~\cite{LP:2025b,PX:2024} to the nonsmooth case, and employs a broader local stability assumption~(see \cref{Ass:local}(c)) than the one used in~\cite{Park:2020}.

\begin{assumption}[local problems]
\label{Ass:local}
For any $j \in [J]$ and $v \in V$, the local functionals $F_j (\cdot ; v) \colon V_j \to \mathbb{R}$ and $G_j (\cdot; v) \colon V_j \to \overline{\mathbb{R}}$ satisfy the following:
\begin{enumerate}[(a)]
\item (convexity) The functional \( F_j(\cdot; v) \) is G\^{a}teaux differentiable and convex, while \( G_j(\cdot; v) \) is proper, convex, and lower semicontinuous.
Moreover, the composite functional \( E_j(\cdot; v) \) is coercive.

\item (consistency) We have
\begin{equation*}
F_j (0;v) = F(v), \quad
G_j (0;v) = G(v),
\end{equation*}
and
\begin{equation*}
\langle F_j' (0;v), w_j \rangle = \langle F'(v), w_j \rangle ,
\quad w_j \in V_j.
\end{equation*}

\item (stability) For some $\omega \in (0, 1] \cup (1, \rho)$, we have
\begin{equation*}
d (w_j; v) \leq \omega d_j (w_j ; v), \quad
G(v + w_j) \leq G_j (w_j; v),
\quad w_j \in V_j,
\end{equation*}
where the constant $\rho$ is defined as
\begin{equation}
\label{rho}
    \rho = \min_{j \in [J]} \inf_{v \in V,\ w_j \in V_j} \frac{\langle d_j' (w_j; v), w_j \rangle}{d_j (w_j; v)}.
\end{equation}
\end{enumerate}
\end{assumption}

The constant \( \rho \) defined in~\eqref{rho} is always greater than or equal to $1$ as a consequence of \cref{Ass:local}(a,b).  
In the case of linear problems, one can verify that $\rho = 2$~\cite[Example~1]{PX:2024}, which is consistent with~\cite{TW:2005,XZ:2002}.
A nonlinear example where \( \rho > 1 \) is provided in~\cite[Example~A.2]{LP:2025b}.  
In \cref{Lem:local_sufficient_decrease}, which is a nonsmooth extension of~\cite[Lemma~1]{PX:2024}, we show that \cref{Ass:local} ensures that solving each local problem leads to a decrease in the global energy.

\begin{lemma}
\label{Lem:local_sufficient_decrease}
For $j \in [J]$ and $v \in V$, let
\begin{equation}
\label{w_hat}
\hat{w}_j \in \operatornamewithlimits{\arg\min}_{w_j \in V_j} E_j (w_j; v).
\end{equation}
Under \cref{Ass:local}, we have
\begin{equation*}
E(v) - E(v + \hat{w}_j) \geq \left( 1 - \frac{\omega}{\rho} \right) \langle d_j' (\hat{w}_j; v), \hat{w}_j \rangle \geq 0.
\end{equation*}
\end{lemma}
\begin{proof}
The optimality condition for \( \hat{w}_j \) reads as
\begin{equation}
\label{Lem1:local_sufficient_decrease}
G_j(w_j; v) - G_j(\hat{w}_j; v) \geq \langle F_j'(\hat{w}_j; v), \hat{w}_j - w_j \rangle,
\quad w_j \in V_j.
\end{equation}
In particular, for \( w_j = 0 \), we obtain
\begin{equation}
\label{Lem2:local_sufficient_decrease}
\begin{split}
G(v) - G(v + \hat{w}_j) 
&\geq G_j(0; v) - G_j(\hat{w}_j; v) \\
&\stackrel{\eqref{Lem1:local_sufficient_decrease}}{\geq} \langle F_j'(\hat{w}_j; v), \hat{w}_j \rangle \\
&= \langle F'(v), \hat{w}_j \rangle + \langle d_j'(\hat{w}_j; v), \hat{w}_j \rangle,
\end{split}
\end{equation}
where the first inequality follows from \cref{Ass:local}(b,c).
On the other hand, by \cref{Ass:local}(c), we have
\begin{equation}
\label{Lem3:local_sufficient_decrease}
\begin{split}
F(v) - F(v + \hat{w}_j)
&= - \langle F'(v), \hat{w}_j \rangle - d(\hat{w}_j; v) \\
&\geq - \langle F'(v), \hat{w}_j \rangle - \omega d_j(\hat{w}_j; v) \\
&\stackrel{\eqref{rho}}{\geq} - \langle F'(v), \hat{w}_j \rangle - \frac{\omega}{\rho} \langle d_j'(\hat{w}_j; v), \hat{w}_j \rangle.
\end{split}
\end{equation}
Summing~\eqref{Lem2:local_sufficient_decrease} and~\eqref{Lem3:local_sufficient_decrease} completes the proof.
\end{proof}

As a corollary, the energy sequence generated by \cref{Alg:RSC} decreases monotonically; see \cref{Cor:decrease}.

\begin{corollary}
\label{Cor:decrease}
Suppose that \cref{Ass:local} holds.
In the randomized subspace correction method~(\cref{Alg:RSC}), the sequence $\{ E (u^{(n)}) \}$ is decreasing.
\end{corollary}

In \cref{Lem:ASM}, we present a refined version~(cf.~\cite[Lemma~2]{PX:2024}) of the generalized additive Schwarz lemma for the composite optimization problem~\eqref{model}, originally introduced in~\cite[Lemma~4.5]{Park:2020}.

\begin{lemma}
\label{Lem:ASM}
Suppose that \cref{Ass:local}(a,b) holds.
For $v \in V$, we have
\begin{equation}
\label{Lem1:ASM}
\hat{w} := \sum_{j=1}^J \hat{w}_j \in \operatornamewithlimits{\arg\min}_{w \in V} \left\{ \langle F'(v), w \rangle + \inf_{w = \sum_{j=1}^J w_j} \sum_{j=1}^J (d_j + G_j)(w_j; v) \right\},
\end{equation}
where $\hat{w}_j$, $j \in [J]$, were given in~\eqref{w_hat}.
Moreover, we have
\begin{equation}
\label{Lem2:ASM}
\inf_{w = \sum_{j=1}^J w_j} \sum_{j=1}^J (d_j + G_j)(w_j; v)
= \sum_{j=1}^J (d_j + G_j)(\hat{w}_j; v).
\end{equation}
\end{lemma}
\begin{proof}
We closely follow the argument in the proof of~\cite[Lemma~4.2]{LP:2025b}.
Throughout the proof, we define
\begin{equation*}
d_{\mathrm{PSC}} (w; v) = \inf_{w = \sum_{j=1}^J w_j} \sum_{j=1}^J (d_j + G_j)(w_j; v),
\quad w \in V.
\end{equation*}
Let $w \in V$ be arbitrary.
For any $w_j \in V_j$, $j \in [J]$, such that $w = \sum_{j=1}^J w_j$, we have
\begin{equation}
\label{Lem3:ASM}
\begin{split}
\langle F'(v), \hat{w} \rangle + d_{\mathrm{PSC}} (\hat{w}; v)
&\leq \sum_{j=1}^J \left( \langle F'(v), \hat{w}_j \rangle + (d_j + G_j)(\hat{w}_j; v) \right) \\
&\stackrel{\eqref{w_hat}}{\leq} \sum_{j=1}^J \left( \langle F'(v), w_j \rangle + (d_j + G_j)(w_j; v) \right) \\
&= \langle F'(v), w \rangle + \sum_{j=1}^J (d_j + G_j)(w_j; v).
\end{split}
\end{equation}
Here, the first inequality and the last equality follow from the fact that
$\hat{w} = \sum_{j=1}^J \hat{w}_j \in \sum_{j=1}^J V_j$.
By minimizing the last line of~\eqref{Lem3:ASM} over all decompositions $(w_j)_{j=1}^J$, we obtain
\begin{equation}
\label{Lem4:ASM}
\begin{split}
\langle F'(v), \hat{w} \rangle + d_{\mathrm{PSC}} (\hat{w}; v)
&\leq \sum_{j=1}^J \left( \langle F'(v), \hat{w}_j \rangle + (d_j + G_j)(\hat{w}_j; v) \right) \\
&\leq \langle F'(v), w \rangle + d_{\mathrm{PSC}} (w; v),
\end{split}
\end{equation}
which implies~\eqref{Lem1:ASM}.
Finally, setting $w = \hat{w}$ in~\eqref{Lem4:ASM} yields~\eqref{Lem2:ASM}.
\end{proof}

\cref{Lem:ASM} shows that, to analyze the convergence rate of the parallel subspace correction method~(\cref{Alg:PSC}), it suffices to estimate the following quantity~\cite{LP:2025b,PX:2024}:
\begin{equation}
\label{Psi}
\Psi(u^{(n)}) :=
\min_{w \in V} \left\{ \langle F'(u^{(n)}), w \rangle + \inf_{w = \sum_{j=1}^J w_j} \sum_{j=1}^J ( d_j + G_j )(w_j; u^{(n)}) \right\} - J G(u^{(n)}).
\end{equation}

In the randomized subspace correction method~(\cref{Alg:RSC}), the update at each iteration is determined by a randomly chosen subspace.
Consequently, the total energy after one step, $E(u^{(n+1)})$, is a random variable depending on the sampling of the index $j \in [J]$.
To analyze its expected descent behavior, we consider the conditional expectation of the energy $\mathbb{E} [ E (u^{(n+1)}) \mid u^{(n)} ]$, which represents the expected value of the energy at the next iteration given the current iterate $u^{(n)}$.
Here, the roman font $E$ denotes the objective functional, while $\mathbb{E}$ denotes expectation.
In \cref{Thm:decrease}, we show that the conditional expectation $\mathbb{E} [ E (u^{(n+1)}) \mid u^{(n)} ]$ can be estimated using~\eqref{Psi}, indicating that its analysis can proceed along similar lines as that of the parallel method.

\begin{theorem}
\label{Thm:decrease}
Suppose that \cref{Ass:local} holds.
In the randomized subspace correction method~(\cref{Alg:RSC}), we have
\begin{equation*}
\mathbb{E} [ E (u^{(n+1)}) \mid u^{(n)} ] \leq E (u^{(n)}) + \frac{\theta}{J} \Psi (u^{(n)}), \quad n \geq 0,
\end{equation*}
where $\Psi (u^{(n)})$ was given in~\eqref{Psi}, and the constant $\theta$ is given by
\begin{equation}
\label{theta}
\theta = \begin{cases}
1, \quad & \text{ if } \omega \in (0, 1], \\
\frac{\rho - \omega}{\rho - 1}, \quad & \text{ if } \omega \in (1, \rho).
\end{cases}
\end{equation}
\end{theorem}
\begin{proof}
Fix any \( n \geq 0 \).
For each \( j \in [J] \), let \( w_j^{(n+1)} \in V_j \) be a minimizer of \( E_j(w_j; u^{(n)}) \).
We first consider the case \( \omega \in (0, 1] \).
It follows that
\begin{equation}
\label{Thm1:decrease}
\begin{split}
\mathbb{E}[ E(u^{(n+1)}) \mid u^{(n)} ]
&= \frac{1}{J} \sum_{j=1}^J E(u^{(n)} + w_j^{(n+1)}) \\
&\stackrel{\text{(i)}}{\leq} F(u^{(n)}) + \frac{1}{J} \sum_{j=1}^J \left[ \langle F'(u^{(n)}), w_j^{(n+1)} \rangle + (d_j + G_j)(w_j^{(n+1)}; u^{(n)}) \right] \\
&\stackrel{\text{(ii)}}{=} E(u^{(n)}) + \frac{1}{J} \Psi(u^{(n)}),
\end{split}
\end{equation}
which is the desired result. Here, (i) follows from \cref{Ass:local}(c), and (ii) follows from \cref{Lem:ASM}.

Now consider the case \( \omega \in (1, \rho) \). Proceeding similarly as in~\eqref{Thm1:decrease}, we obtain
\begin{equation}
\label{Thm2:decrease}
\mathbb{E}[ E(u^{(n+1)}) \mid u^{(n)} ]
\leq E(u^{(n)}) 
+ \frac{1}{J} \Psi(u^{(n)})
+ \frac{\omega - 1}{J} \sum_{j=1}^J d_j(w_j^{(n+1)}; u^{(n)}).
\end{equation}
Meanwhile, from~\eqref{rho} and \cref{Lem:local_sufficient_decrease}, it follows that
\begin{equation}
\label{Thm3:decrease}
\begin{split}
\frac{1}{J} \sum_{j=1}^J d_j(w_j^{(n+1)}; u^{(n)})
&\leq \frac{1}{J \rho} \sum_{j=1}^J \langle d_j'(w_j^{(n+1)}; u^{(n)}), w_j^{(n+1)} \rangle \\
&\leq \frac{1}{\rho - \omega} \left( E(u^{(n)}) - \mathbb{E}[ E(u^{(n+1)}) \mid u^{(n)} ] \right).
\end{split}
\end{equation}
Combining~\eqref{Thm2:decrease} and~\eqref{Thm3:decrease} yields the desired result.
\end{proof}

\begin{remark}
\label{Rem:nonuniform_sampling}
The randomized subspace correction method in \cref{Alg:RSC} can be extended to nonuniform sampling, as considered in, e.g.,~\cite{Nesterov:2012,QR:2014,RT:2016}.
Namely, given a probability distribution $\boldsymbol{\gamma}=\{\gamma_1,\dots,\gamma_J\}$, each subspace index $j$ ($j \in [J]$) is selected with probability $\gamma_j$, where $\gamma_j \in (0,1)$ and $\sum_{j=1}^J \gamma_j = 1$.
In this case, using arguments similar to those in \cref{Lem:ASM} and \cref{Thm:decrease}, we obtain the following counterpart of \cref{Thm:decrease}:
\begin{equation*}
    \mathbb{E}[E(u^{(n+1)}) \mid u^{(n)}] \leq E(u^{(n)}) + \frac{\theta}{J} \Psi_{\boldsymbol{\gamma}}(u^{(n)}),
    \quad n \geq 0,
\end{equation*}
where $\Psi_{\boldsymbol{\gamma}}(u^{(n)})$ is given by (cf.~\eqref{Psi})
\begin{equation*}
\resizebox{\textwidth}{!}{$\displaystyle
    \Psi_{\boldsymbol{\gamma}}(u^{(n)}) = \min_{w \in V} \left\{ \langle F'(u^{(n)}), w \rangle + \inf_{w = \sum_{j=1}^J \gamma_j J w_j} \sum_{j=1}^J \gamma_j J (d_j + G_j)(w_j; u^{(n)}) \right\} - JG(u^{(n)}).
    $}
\end{equation*}
\end{remark}
Such nonuniform sampling is particularly beneficial when the contributions of the $(d_j + G_j)$-terms vary significantly across $j \in [J]$.
In this case, one may choose smaller $\gamma_j$ for indices with larger contributions of $(d_j + G_j)$ to control $\Psi_{\boldsymbol{\gamma}}(u^{(n)})$; see, e.g.,~\cite{AQRY:2016} for specific sampling strategies and the resulting improvements.

\section{Convergence theorems}
\label{Sec:Convergence}
In this section, we present convergence theorems for the randomized subspace correction method~(\cref{Alg:RSC}) under various conditions on the energy functional~$E$.
We note that the convergence analysis of \cref{Alg:RSC} given in this section relies on existing results for the parallel subspace correction method~(\cref{Alg:PSC})~\cite{LP:2025b,Park:2020}. 
In particular, convergence rates of \cref{Alg:PSC} have been established in~\cite[Theorem~4.7]{Park:2020} for general convex problems and in~\cite[Theorem~4.8]{Park:2020} for sharp~(cf.~\cref{Ass:sharp}) convex problems.

Given the initial iterate $u^{(0)} \in \operatorname{dom} G$ of \cref{Alg:RSC}, we define
\begin{equation}
\label{K_0}
K_0 = \{ v \in V : E (v) \leq E (u^{(0)}) \}, \quad
R_0 = \sup_{v \in K_0} \| v - u \|.
\end{equation}
The convexity and coercivity of $E$ imply that $K_0$ is bounded and convex, and in particular, $R_0 < \infty$.
Moreover, by \cref{Cor:decrease}, the sequence $\{ u^{(n)} \}$ generated by \cref{Alg:RSC} remains entirely within $K_0$.

\cref{Thm:decrease} implies that, to estimate the convergence rate of \cref{Alg:RSC}, it suffices to estimate $\Psi(u^{(n)})$ defined in~\eqref{Psi}.
The form of $\Psi(u^{(n)})$ naturally motivates the following stable decomposition assumption~(cf.~\cite[Assumption~4.1]{Park:2020}).

\begin{assumption}[stable decomposition]
\label{Ass:stable}
For some $q > 1$, the following holds: for any bounded convex subset $K$ of $V$, we have
\begin{subequations}
\label{C_K}
\begin{equation}
C_K := q \sup_{v, v+w \in K} \inf \frac{\sum_{j=1}^J d_j (w_j; v)}{\| w \|^q}   < \infty,
\end{equation}
where the infimum is taken over $w_j \in V_j$, $j \in [J]$, satisfying
\begin{equation}
\label{C_K_infimum}
w = \sum_{j=1}^J w_j, \quad
\sum_{j=1}^J G_j (w_j; v) \leq G(v+w) + (J-1) G(v).
\end{equation}
\end{subequations}
\end{assumption}

Examples of stable decompositions satisfying \cref{Ass:stable} will be provided in \cref{Sec:Related,Sec:Applications}; see also~\cite[Section~6]{Park:2020}.
A notable observation made in~\cite[Lemma~3]{PX:2024} is that, in the case of smooth problems, i.e., when $G = 0$ in~\eqref{model} and $G_j = 0$, $j \in [J]$, in~\eqref{local_inexact}, \cref{Ass:stable} need not be assumed, but instead holds automatically under a mild smoothness condition on each $d_j$; see \cref{Prop:stable}.

\begin{proposition}
\label{Prop:stable}
In the case of smooth problems, i.e., when $G = 0$ in~\eqref{model} and $G_j = 0$, $j \in [J]$, in~\eqref{local_inexact}, suppose that \cref{Ass:local} holds.
Furthermore, assume that for some $q > 1$, the following holds: for any bounded convex subsets $K \subset V$ and $K_j \subset V_j$ with $0 \in K_j$, we have
\begin{equation*}
\sup_{v \in K,\ w_j \in K_j} \frac{d_j (w_j; v)}{\| w_j \|^q} < \infty.
\end{equation*}
Then we have
\begin{equation*}
C_{K} = q \sup_{v,\, v+w \in K} \inf_{w = \sum_{j=1}^J w_j} \frac{\sum_{j=1}^J d_j (w_j; v)}{\| w \|^q} < \infty.
\end{equation*}
\end{proposition}
\begin{proof}
This result is a special case of~\cite[Lemma~4.10]{LP:2025b}, which relies on the stable decomposition property~\eqref{norm_stable}.
\end{proof}

The following lemma provides a preliminary estimate for $\Psi(u^{(n)})$ under \cref{Ass:stable}.
Although the proof follows a similar argument to that in~\cite[Appendix~A.3]{Park:2020}, we include it here for completeness.

\begin{lemma}
\label{Lem:Psi}
Suppose that \cref{Ass:stable} holds.
Then we have
\begin{equation}
\label{Lem1:Psi}
\resizebox{\textwidth}{!}{$\displaystyle
\Psi (u^{(n)}) \leq \min_{t \in [0, 1]} \bigg\{ t \langle F'(u^{(n)}), u - u^{(n)} \rangle + \frac{C_{K_0}}{q} t^q \| u - u^{(n)} \|^q
+ G ( (1-t) u^{(n)} + t u ) \bigg\} - G(u^{(n)}),\ n \geq 0,
$}
\end{equation}
where $\Psi (u^{(n)})$, $K_0$, and $C_{K_0}$  were given in~\eqref{Psi},~\eqref{C_K}, and~\eqref{K_0}, respectively.
\end{lemma}
\begin{proof}
From the definition~\eqref{Psi} of $\Psi(u^{(n)})$, we have
\begin{equation}
\label{Lem2:Psi}
\resizebox{\textwidth}{!}{$\displaystyle
\begin{aligned}
\Psi(u^{(n)})
&= \min_{w \in V} \left\{ \langle F'(u^{(n)}), w \rangle + \inf_{w = \sum_{j=1}^J w_j} \sum_{j=1}^J ( d_j + G_j )(w_j; u^{(n)}) \right\} - J G(u^{(n)}) \\
&\leq \min_{u^{(n)} + w \in K_0} \left\{ \langle F'(u^{(n)}), w \rangle + \inf_{w = \sum_{j=1}^J w_j} \sum_{j=1}^J (d_j + G_j)(w_j; u^{(n)}) \right\} - J G(u^{(n)}) \\
&\leq \min_{u^{(n)} + w \in K_0} \left\{ \langle F'(u^{(n)}), w \rangle + \frac{C_{K_0}}{q} \| w \|^q + G(u^{(n)} + w) \right\} - G(u^{(n)}),
\end{aligned}
$}
\end{equation}
where the last inequality follows from \cref{Ass:stable}.
The proof is complete upon replacing $w$ in the last line of~\eqref{Lem2:Psi} with $t(u - u^{(n)})$ for some $t \in [0, 1]$.
\end{proof}

If one can derive suitable bounds for the term
\begin{equation}
\label{Psi_tilde}
\tilde{\Psi}(t; u^{(n)}) := t \langle F'(u^{(n)}), u - u^{(n)} \rangle + G\big((1-t)u^{(n)} + tu\big) - G(u^{(n)}),
\quad t \in [0,1],
\end{equation}
then, in view of \cref{Lem:Psi}, one can proceed to obtain quantitative convergence bounds.
Sharper estimates for $\tilde{\Psi}(t; u^{(n)})$ can be established under stronger assumptions on $F$ and $G$.
Accordingly, we consider several such cases in the following subsections.

\begin{remark}
\label{Rem:Psi}
An improved estimate for $\Psi(u^{(n)})$ compared to that in \cref{Lem:Psi} can be obtained under a stronger assumption than \cref{Ass:stable}.
Suppose that the following \textit{global} stable decomposition condition holds:
\begin{equation}
\label{C_V}
C_V := q \sup_{v,\, v+w \in V} \inf \frac{\sum_{j=1}^J d_j(w_j; v)}{\| w \|^q} < \infty,
\end{equation}
where the infimum is taken over~\eqref{C_K_infimum}.
Under this global condition, the constraint $u^{(n)} + w \in K_0$ in~\eqref{Lem2:Psi} is no longer needed, and consequently, the restriction $t \in [0,1]$ in~\eqref{Lem1:Psi} can be relaxed to $t \geq 0$.
This improvement will be useful later in our analysis; see \cref{Rem:conv_sharp,Rem:conv_strong}.
\end{remark}

\subsection{General problems}
Without imposing additional assumptions on $F$ and $G$, we can still obtain the following upper bound for~\eqref{Psi_tilde} using the convexity of $F$ and $G$:
\begin{equation}
\label{Psi_tilde_bound}
\tilde{\Psi} (t; u^{(n)}) \leq - t ( E(u^{(n)}) - E(u)).
\end{equation}
By combining \cref{Thm:decrease}, \cref{Lem:Psi}, and~\eqref{Psi_tilde_bound}, we obtain the following convergence theorem for the randomized subspace correction method~(\cref{Alg:RSC}).

\begin{theorem}
\label{Thm:conv}
Suppose that \cref{Ass:local,Ass:stable} hold.
In the randomized subspace correction method~(\cref{Alg:RSC}), if $\zeta_0 := E (u^{(0)}) - E (u) > C_{K_0} R_0^q$, then
\begin{equation*}
    \mathbb{E}[E (u^{(1)})] - E (u) \leq \left(1 - \frac{\theta}{J} \left( 1 - \frac{1}{q} \right) \right) \zeta_0,
\end{equation*}
where $\theta$, $K_0$, $R_0$, and $C_{K_0}$ were given in~\eqref{theta}, \eqref{K_0}, and~\eqref{C_K}.
Otherwise, we have
\begin{equation*}
    \mathbb{E}[E (u^{(n)})] - E (u) \leq \frac{C}{ ( n + (C / \zeta_0)^{1/\beta} )^{\beta}},
    \quad n \geq 0,
\end{equation*}
where
\begin{equation*}
\beta = q-1, \quad
C = \left( \frac{J q}{\theta} \right)^{q-1} C_{K_0} R_0^q.
\end{equation*}
\end{theorem}
\begin{proof}
We write $\zeta_n = E(u^{(n)}) - E(u)$.
Combining \cref{Lem:Psi} and~\eqref{Psi_tilde_bound} yields
\begin{equation}
\label{Thm1:conv}
\Psi (u^{(n)}) \leq \min_{t \in [0,1]} \left\{ -t \zeta_n + \frac{t^q C_{K_0}}{q} \| u - u^{(n)} \|^q \right\}.
\end{equation}
Applying the argument in~\cite[Equation~(B.2)]{LP:2025b}, we obtain
\begin{equation}
\label{Thm2:conv}
\Psi (u^{(n)})
\leq \begin{cases}
- \left(1 - \frac{1}{q} \right) \zeta_n & \text{ if } \zeta_n > C_{K_0} R_0^q, \\
- \left(1 - \frac{1}{q} \right) \frac{\zeta_n^{\frac{q}{q-1}}}{(C_{K_0} R_0^q)^{\frac{1}{q-1}}} & \text{ if } \zeta_n \leq C_{K_0} R_0^q.
\end{cases}
\end{equation}
Combining \cref{Thm:decrease} and~\eqref{Thm2:conv}, we obtain
\begin{equation}
\label{Thm3:conv}
\mathbb{E} [E(u^{(n+1)}) \mid u^{(n)}] - E(u)
\leq \begin{cases}
\left( 1 - \frac{\theta}{J} \left(1 - \frac{1}{q} \right) \right) \zeta_n & \text{ if } \zeta_n > C_{K_0} R_0^q, \\
\zeta_n - \frac{\theta}{J} \left(1 - \frac{1}{q} \right) \frac{\zeta_n^{\frac{q}{q-1}}}{(C_{K_0} R_0^q)^{\frac{1}{q-1}}} & \text{ if } \zeta_n \leq C_{K_0} R_0^q.
\end{cases}
\end{equation}
This proves the desired result for the case $\zeta_0 > C_{K_0} R_0^q$.
On the other hand, by \cref{Cor:decrease}, the condition $\zeta_0 \leq C_{K_0} R_0^q$ ensures $\zeta_n \leq C_{K_0} R_0^q$.
By the law of total expectation and Jensen inequality
\begin{equation*}
\mathbb{E} [\zeta_n]^{\frac{q}{q-1}} \leq \mathbb{E} [\zeta_n^{\frac{q}{q-1}}],
\end{equation*}
we obtain
\begin{equation*}
\mathbb{E} [\zeta_{n+1}]
\leq \mathbb{E}[\zeta_n] - \frac{\theta}{J} \left(1 - \frac{1}{q} \right) \frac{\mathbb{E} [\zeta_n]^{\frac{q}{q-1}}}{(C_{K_0} R_0^q)^{\frac{1}{q-1}}}
\end{equation*}
if $\zeta_0 \leq C_{K_0} R_0^q$.
Finally, invoking~\cite[Lemma~B.2]{LP:2025b} completes the proof for the case $\zeta_0 \leq C_{K_0} R_0^q$.
\end{proof}

The generality of the assumptions in \cref{Thm:conv} enables a broad range of applications, particularly in scenarios where $F$ exhibits a weaker level of smoothness than the standard smoothness condition~\cite{Park:2022b}; see~\cite{LP:2025b,Nesterov:2015} for concrete examples.

\begin{remark}
\label{Rem:J}
Since \cref{Alg:RSC} is expected to visit all subspaces $\{ V_j \}_{j=1}^J$ on average within $J$ iterations, it is natural to examine the convergence behavior of \cref{Alg:RSC} at iteration counts that are integer multiples of $J$, say $nJ$.
We observe that the expected energy error has an upper bound independent of $J$:
\begin{equation*}
\mathbb{E} \bigl[ E(u^{(nJ)}) \bigr] - E(u) \leq \frac{\hat{C}}{ ( n + ( \hat{C} / \zeta_0 )^{1/\beta} )^{\beta}}, \quad n \geq 0,
\end{equation*}
where
\begin{equation*}
\beta = q - 1, \qquad
\hat{C} = \left( \frac{q}{\theta} \right)^{q-1} C_{K_0} R_0^q.
\end{equation*}
\end{remark}

\subsection{Sharp problems}
Meanwhile, in many applications, the energy functional $E$ satisfies the sharpness condition~\cite{RD:2020}, summarized in \cref{Ass:sharp}, which is also known as the H\"{o}lderian error bound of the \L{}ojasiewicz inequality~\cite{BDL:2007,XY:2013}.

\begin{assumption}[sharpness]
\label{Ass:sharp}
For some $p > 1$, the function $E$ satisfies the following: for any bounded convex subset $K$ of $V$ satisfying $u \in K$, we have
\begin{equation}
\label{mu_K}
\mu_K := p \inf_{v \in K} \frac{E(v) - E(u)}{\| v - u \|^p} > 0.
\end{equation}
\end{assumption}

If we additionally assume that \cref{Ass:sharp} holds, then we can derive the following improved convergence theorem for the randomized subspace correction method.

\begin{theorem}
\label{Thm:conv_sharp}
Suppose that \cref{Ass:local,Ass:stable,Ass:sharp} hold.
In the randomized subspace correction method~(\cref{Alg:RSC}), we have the following:
\begin{enumerate}[(a)]
\item In the case $p = q$, we have
\begin{equation*}
\mathbb{E}[ E (u^{(n)}) ] - E (u) 
\leq \left( 1 - \frac{\theta}{J} \left(1 - \frac{1}{q} \right) \min \left\{ 1, \frac{\mu_{K_0}}{q C_{K_0}} \right\}^{\frac{1}{q-1}} \right)^{n} \zeta_0,
\quad n \geq 0,
\end{equation*}
where $\zeta_0 = E (u^{(0)}) - E (u)$, and $\theta$, $K_0$, $C_{K_0}$, and $\mu_{K_0}$ were given in~\eqref{theta}, \eqref{K_0}, \eqref{C_K} and \eqref{mu_K}, respectively.

\item In the case $p > q$, if $\zeta_0 > \left( \frac{p}{\mu_{K_0}} \right)^{\frac{q}{p-q}} C_{K_0}^{\frac{p}{p-q}}$, then we have
\begin{equation*}
    \mathbb{E}[E (u^{(1)})] - E (u) \leq \left(1 - \frac{\theta}{J} \left( 1 - \frac{1}{q} \right) \right) \zeta_0.
\end{equation*}
Otherwise, we have
\begin{equation*}
    \mathbb{E}[E (u^{(n)})] - E (u) \leq \frac{C}{\left( n + (C / \zeta_0)^{1/\beta} \right)^{\beta}},
    \quad n \geq 0,
\end{equation*}
where
\begin{equation*}
\beta = \frac{p(q-1)}{p-q}, \quad
C = \left( \frac{ J p q }{(p-q) \theta} \right)^{\frac{p(q-1)}{p-q}} \left( \frac{p}{\mu_{K_0}} \right)^{\frac{q}{p-q}} C_{K_0}^{\frac{p}{p-q}}.
\end{equation*}
\end{enumerate}
\end{theorem}
\begin{proof}
We again set $\zeta_n = E(u^{(n)}) - E(u)$.
By \cref{Ass:sharp} and~\eqref{Thm1:conv}, we obtain
\begin{equation}
\label{Thm1:conv_sharp}
\Psi (u^{(n)})
\leq \min_{t \in [0,1]} \left\{ - t \zeta_n + \frac{t^q C_{K_0}}{q} \| u - u^{(n)} \|^q \right\}
\leq \min_{t \in [0,1]} \left\{ - t \zeta_n + \frac{t^q p^{\frac{q}{p}} C_{K_0}}{q \mu_{K_0}^{\frac{q}{p}}} \zeta_n^{\frac{q}{p}} \right\}.
\end{equation}

We first consider the case $p = q$.
It follows from~\eqref{Thm1:conv_sharp} that
\begin{equation}
\label{Thm2:conv_sharp}
\Psi (u^{(n)})
\leq \min_{t \in [0,1]} \left\{ - t \zeta_n + \frac{t^q q C_{K_0}}{q \mu_{K_0}} \zeta_n \right\}
\leq - \zeta_n \left(1 - \frac{1}{q} \right)
\min \left\{ 1 , \frac{\mu_{K_0}}{q C_{K_0}} \right\}^{\frac{1}{q-1}}.
\end{equation}
Combining \cref{Thm:decrease} and~\eqref{Thm2:conv_sharp}, we obtain
\begin{equation*}
\mathbb{E} [E(u^{(n+1)}) \mid u^{(n)} ] - E(u)
\leq \left( 1 - \frac{\theta}{J} \left(1 - \frac{1}{q} \right)
\min \left\{ 1, \frac{\mu_{K_0}}{q C_{K_0}} \right\}^{\frac{1}{q-1}} \right) \zeta_n.
\end{equation*}
Invoking the law of total expectation yields the desired result.

Next, we consider the case $p > q$.
Applying the argument in~\cite[Equation~(B.5)]{LP:2025b} to~\eqref{Thm1:conv_sharp}, we obtain
\begin{equation}
\label{Thm3:conv_sharp}
\Psi (u^{(n)})
\leq \begin{cases}
- \left(1 - \frac{1}{q} \right) \zeta_n
& \text{ if } \zeta_n > \left( \frac{p}{\mu_{K_0}} \right)^{\frac{q}{p-q}} C_{K_0}^{\frac{p}{p-q}}, \\[0.5ex]
- \left(1 - \frac{1}{q} \right)
\left( \frac{\mu_{K_0}}{p} \right)^{\frac{q}{p(q-1)}}
\frac{\zeta_n^{\frac{q(p-1)}{p(q-1)}}}{C_{K_0}^{\frac{1}{q-1}}}
& \text{ if } \zeta_n \leq \left( \frac{p}{\mu_{K_0}} \right)^{\frac{q}{p-q}} C_{K_0}^{\frac{p}{p-q}}.
\end{cases}
\end{equation}
Combining \cref{Thm:decrease} and~\eqref{Thm3:conv_sharp}, we derive
\begin{equation}
\label{Thm4:conv_sharp}
\resizebox{\textwidth}{!}{$\displaystyle
\mathbb{E} [ E (u^{(n+1)}) \mid u^{(n)} ]
\leq \begin{cases}
\left( 1 - \frac{\theta}{J} \left(1 - \frac{1}{q} \right) \right) \zeta_n
& \text{ if } \zeta_n > \left( \frac{p}{\mu_{K_0}} \right)^{\frac{q}{p-q}} C_{K_0}^{\frac{p}{p-q}}, \\[0.5ex]
\zeta_n - \frac{\theta}{J} \left(1 - \frac{1}{q} \right)
\left( \frac{\mu_{K_0}}{p} \right)^{\frac{q}{p(q-1)}}
\frac{\zeta_n^{\frac{q(p-1)}{p(q-1)}}}{C_{K_0}^{\frac{1}{q-1}}}
& \text{ if } \zeta_n \leq \left( \frac{p}{\mu_{K_0}} \right)^{\frac{q}{p-q}} C_{K_0}^{\frac{p}{p-q}}.
\end{cases}
$}
\end{equation}
Observing that~\eqref{Thm3:conv} and~\eqref{Thm4:conv_sharp} share the same structure, proceeding as in the proof of \cref{Thm:conv} completes the proof.
\end{proof}

\begin{remark}
\label{Rem:J_sharp}
Similar to \cref{Rem:J}, the expected energy error at the $nJ$th iteration of \cref{Alg:RSC} admits an upper bound that is independent of $J$.
In the case $p = q$, we have
\begin{equation*}
\begin{split}
\mathbb{E}[E(u^{(n J)})] - E(u)
&\leq \left( 1 - \frac{\theta}{J} \left( 1 - \frac{1}{q} \right) \min \left\{ 1, \frac{\mu_{K_0}}{q C_{K_0}}  \right\}^{\frac{1}{q-1}} \right)^{n J} \zeta_0 \\
&\leq \exp \left( - n \theta \left( 1 - \frac{1}{q} \right) \min \left\{ 1, \frac{\mu_{K_0}}{q C_{K_0}} \right\}^{\frac{1}{q-1}} \right) \zeta_0.
\end{split}
\end{equation*}
The case $p > q$ can be analyzed analogously to \cref{Rem:J}. 
Namely, we have
\begin{equation*}
\mathbb{E} \bigl[ E(u^{(nJ)}) \bigr] - E(u) \leq \frac{\hat{C}}{ ( n + ( \hat{C} / \zeta_0 )^{1/\beta} )^{\beta}}, \quad n \geq 0,
\end{equation*}
where
\begin{equation*}
\beta = \frac{p(q-1)}{p-q}, \quad
\hat{C} = \left( \frac{ p q }{(p-q) \theta} \right)^{\frac{p(q-1)}{p-q}} \left( \frac{p}{\mu_{K_0}} \right)^{\frac{q}{p-q}} C_{K_0}^{\frac{p}{p-q}}.
\end{equation*}
\end{remark}

\begin{remark}
\label{Rem:conv_sharp}
In the case of smooth problems, i.e., when $G = 0$ in~\eqref{model} and $G_j = 0$, $j \in [J]$ in~\eqref{local_inexact}, the global stable decomposition condition~\eqref{C_V}, together with the argument in \cref{Rem:Psi}, yields the following simplified estimate corresponding to \cref{Thm:conv_sharp}(a):
\begin{equation*}
\mathbb{E} \big[F(u^{(n)})\big] - F(u) \leq \left( 1 - \frac{\theta}{J} \left(1 - \frac{1}{q} \right) \left( \frac{\mu_{K_0}}{q C_V} \right)^{\frac{1}{q-1}} \right)^{n} \zeta_0, \quad n \geq 0.
\end{equation*}
\end{remark}

\subsection{Strongly convex problems}
Next, we consider a stronger condition than \cref{Ass:sharp}, stated in \cref{Ass:strong}.
Recall that a functional $H \colon V \to \overline{\mathbb{R}}$ is said to be $\mu$-strongly convex on a convex set $K \subset V$ if
\begin{equation*}
H ((1-t)v + tw) \leq (1-t) H(v) + t H(w) - t(1-t) \frac{\mu}{2} \| v - w \|^2,
\quad t \in [0,1],\ v,w \in K.
\end{equation*}
Note that \cref{Ass:strong} implies \cref{Ass:sharp} with $p = 2$.

\begin{assumption}[strong convexity]
\label{Ass:strong}
For any bounded convex subset $K$ of $V$ satisfying $u \in K$, $E$ and $F$ are $\mu_K$- and $\mu_{F,K}$-strongly convex on $K$, respectively, for some $\mu_K > 0$ and $\mu_{F,K} \geq 0$.
\end{assumption}

Note that $\mu_K\geq \mu_{F,K}$. Under \cref{Ass:strong}, we have the following upper bound for~\eqref{Psi_tilde}:
\begin{equation}
\label{Psi_tilde_bound_strong}
\begin{split}
\tilde{\Psi} (t; u^{(n)})
&\leq E( (1-t) u^{(n)} + t u) - E(u^{(n)}) - \frac{\mu_{F, K_0}}{2} t^2 \| u - u^{(n)} \|^2 \\
&\leq -t ( E(u^{(n)}) - E(u)) + \left( \frac{\mu_{K_0} - \mu_{F, K_0}}{2} t^2 - \frac{\mu_{K_0}}{2} t \right) \| u - u^{(n)} \|^2,
\end{split}
\end{equation}
where the first and second inequalities follow from the strong convexity of $F$ and $E$, respectively.

As \cref{Ass:strong} imposes a stronger condition than \cref{Ass:sharp}, the convergence rate established in \cref{Thm:conv_sharp}(b) is guaranteed under \cref{Ass:strong}.
However, by using~\eqref{Psi_tilde_bound_strong}, we can derive an even sharper estimate for the convergence rate, as presented in \cref{Thm:conv_strong}, when $q = 2$ in \cref{Ass:stable}.
A similar result appears in~\cite[Theorem~3.3]{Park:2024a}.

\begin{theorem}
\label{Thm:conv_strong}
Suppose that \cref{Ass:local,Ass:stable,Ass:strong} hold with $q = 2$.
In the randomized subspace correction method~(\cref{Alg:RSC}), we have
\begin{equation*}
\mathbb{E}[E (u^{(n)})] - E (u)
\leq \left( 1 - \frac{\theta}{J} \min \left\{ 1,  \frac{\mu_{K_0}}{C_{K_{0}} + \mu_{K_0} - \mu_{F, K_0}} \right\} \right)^{n} \zeta_0, \quad n \geq 0,
\end{equation*}
where $\zeta_0 = E (u^{(0)}) - E (u)$, and $\theta$, $K_0$, $C_{K_0}$, $\mu_{K_0}$, and $\mu_{F, K_0}$ were given in~\eqref{theta}, \eqref{K_0}, \eqref{C_K} and \cref{Ass:strong}.
\end{theorem}
\begin{proof}
Note that $C_{K_0} + \mu_{K_0} - \mu_{F, K_0}$ is positive under \cref{Ass:strong}.
Combining \cref{Lem:Psi} and~\eqref{Psi_tilde_bound_strong} yields
\begin{equation}
\label{Thm1:conv_strong}
\resizebox{\textwidth}{!}{$ \displaystyle
\Psi(u^{(n)}) \leq \min_{t \in [0, 1]} \left\{ -t (E (u^{(n)}) - E(u) ) + \left( \frac{C_{K_0} + \mu_{K_0} - \mu_{F, K_0}}{2} t^2 - \frac{\mu_{K_0}}{2} t \right) \| u - u^{(n)} \|^2 \right\}.
$}
\end{equation}
If $\frac{\mu_{K_0}}{C_{K_0} + \mu_{K_0} - \mu_{F, K_0}} \leq 1$, then setting $t = \frac{\mu_{K_0}}{C_{K_0} + \mu_{K_0} - \mu_{F, K_0}}$ in~\eqref{Thm1:conv_strong} gives
\begin{equation}
\label{Thm2:conv_strong}
\Psi(u^{(n)}) \leq - \frac{\mu_{K_0}}{C_{K_0} + \mu_{K_0} - \mu_{F, K_0}} (E(u^{(n)}) - E(u)).
\end{equation}
Otherwise, if $\frac{\mu_{K_0}}{C_{K_0} + \mu_{K_0} - \mu_{F, K_0}} > 1$, i.e., if $C_{K_0} - \mu_{F, K_0} < 0$, then setting $t = 1$ in~\eqref{Thm1:conv_strong} yields
\begin{equation}
\label{Thm3:conv_strong}
\Psi(u^{(n)})
\leq - (E (u^{(n)}) - E (u)) + \frac{C_{K_0} - \mu_{F, K_0}}{2} \| u - u^{(n)} \|^2
\leq - (E (u^{(n)}) - E (u)).
\end{equation}
Combining \cref{Thm:decrease} with~\eqref{Thm2:conv_strong} and~\eqref{Thm3:conv_strong} completes the proof.
\end{proof}

\begin{remark}
\label{Rem:conv_strong}
Similar to \cref{Rem:conv_sharp}, in the case of smooth problems, i.e., when $G = 0$ in~\eqref{model} and $G_j = 0$, $j \in [J]$ in~\eqref{local_inexact}, the linear convergence rate in~\cref{Thm:conv_strong} simplifies under the global stable decomposition condition~\eqref{C_V} as follows:
\begin{equation*}
\mathbb{E}[ F (u^{(n)}) ] - F (u)
\leq \left( 1 - \frac{\theta\mu_{K_0}}{JC_V} \right)^{n} \zeta_0,
\quad n \geq 0.
\end{equation*}
\end{remark}

\begin{remark}
\label{Rem:overlap}
In all convergence results presented in this section, the convergence rate of \cref{Alg:RSC} improves when $C_{K_0}$ is small, where $C_{K_0}$ was defined in~\eqref{K_0} and~\eqref{C_K}.
Since $C_{K_0}$ involves an infimum over all $w_j \in V_j$, $j \in [J]$, satisfying $w = \sum_{j=1}^J w_j$ for a given $w \in V$, this explains why overlapping subspaces are beneficial: overlap allows more flexibility in the decomposition, leading to a smaller value of the infimum.
In contrast, for nonoverlapping decompositions, the representation $w = \sum_{j=1}^J w_j$ is unique, so the infimum offers no advantage.
\end{remark}

\section{Derivation of related methods}
\label{Sec:Related}
To highlight the versatility of the randomized subspace correction framework presented in this paper, we demonstrate in this section how several related results can be derived from it, including the randomized subspace correction method for linear problems~\cite{GO:2012,HXZ:2019}, block coordinate descent methods~\cite{LX:2015,Nesterov:2012,RT:2014}, and operator splitting methods~\cite{JPX:2025}.

Throughout this section, we assume that $V$ is a Hilbert space equipped with an inner product $(\cdot, \cdot)$ and the induced norm $\| \cdot \|$.
Moreover, we identify $V$ with its topological dual space $V^*$ (cf.~\cite{XZ:2017}).

\subsection{Linear problems}
\label{Subsec:linear}
The randomized subspace correction method for linear problems has been previously studied in~\cite{GO:2012,HXZ:2019}.
Here, we demonstrate how the general framework introduced in this paper recovers these existing results.

In~\eqref{model}, we set
\begin{equation*}
    F(v) = \frac{1}{2} (Av, v) - (f, v), 
    \quad G(v) = 0, 
    \quad v \in V,
\end{equation*}
where \( A \colon V \to V \) is a linear operator induced by a continuous, symmetric, and coercive bilinear form on \( V \), and \( f \in V \).
Then, it is readily observed that~\eqref{model} reduces to the linear problem
\begin{equation}
\label{linear}
    Au = f.
\end{equation}

In the local problem~\eqref{local_inexact}, for each \( j \in [J] \), we set
\begin{equation}
\label{local_linear}
    F_j (w_j; v) = F(v) + (F'(v), w_j) + \frac{1}{2} (R_j^{-1} w_j, w_j), \quad
    G_j (w_j; v) = 0, \quad
    w_j \in V_j,\ v \in V,
\end{equation}
where \( R_j \colon V_j \to V_j \) is a linear operator induced by a continuous, symmetric, and coercive bilinear form on \( V_j \).
Then, the solution \( \hat{w}_j \) of~\eqref{local_inexact} is given by
\begin{equation*}
    \hat{w}_j = - R_j Q_j (Av - f),
\end{equation*}
where \( Q_j \colon V \to V_j \) denotes the orthogonal projection onto \( V_j \).
Hence, we observe that \cref{Alg:RSC} reduces to the randomized subspace correction method for linear problems introduced in~\cite{GO:2012,HXZ:2019}.

To analyze the algorithm, it suffices to verify \cref{Ass:local,Ass:stable,Ass:strong}.
We adopt the \( \| \cdot \|_A \)-norm defined by \( \| \cdot \|_A := (A \cdot, \cdot)^{\frac{1}{2}} \).
Note that \cref{Ass:local}(a, b) are trivially satisfied due to~\eqref{local_linear}.
In \cref{Ass:local}(c), we have \( \rho = 2 \)~\cite[Example~1]{PX:2024}, and the assumption reduces to the following condition: for some \( \omega \in (0, 2) \), we require
\begin{equation}
\label{local_linear_stable}
    (A w_j, w_j) \leq \omega (R_j^{-1} w_j, w_j),
    \quad w_j \in V_j,
\end{equation}
which corresponds to the standard assumption on local solvers, as found in, e.g.,~\cite[Assumption~2.4]{TW:2005} and~\cite[Equation~(4.6)]{XZ:2002}.
Moreover, \cref{Prop:stable} ensures that \cref{Ass:stable} holds.
Specifically, we have \( q = 2 \), and for any bounded convex subset $K$ of $V$, it follows that
\begin{equation*}
    C_K \leq \sup_{w \in V} \inf_{w = \sum_{j=1}^J w_j} 
    \frac{\sum_{j=1}^J (R_j^{-1} w_j, w_j)}{(A w, w)} 
    = \lambda_{\min}(T)^{-1},
\end{equation*}
where the operator \( T \colon V \to V \) is given by
\begin{equation}
\label{T}
    T = \sum_{j=1}^J R_j Q_j A,
\end{equation}
and the last equality follows from the well-known estimate in~\cite[Equation~(2.17)]{XZ:2002}; see also~\cite[Section~4.1]{Park:2020}.
Finally, we observe that \cref{Ass:strong} holds with
\begin{equation*}
\mu_K = \mu_{F,K} \geq 1
\end{equation*}
for any bounded convex subset $K$ of $V$.
Therefore, under the condition~\eqref{local_linear_stable}, the convergence estimate in \cref{Rem:conv_strong} yields
\begin{equation*}
    \mathbb{E}[\| u^{(n)} - u \|_A^2] 
    \leq \left( 1 - \frac{\theta \lambda_{\min}(T)}{J} \right)^{n} 
    \| u^{(0)} - u \|_A^2,
    \quad n \geq 0,
\end{equation*}
where \( \theta \) and $T$ were given in~\eqref{theta} and~\eqref{T}, respectively~(cf.~\cite[Theorem~1(b)]{GO:2012}).


\subsection{Block coordinate descent methods}
\label{Subsec:BCD}
Block coordinate descent methods are important instances of subspace correction methods for convex optimization.
We discuss how randomized block coordinate descent methods~\cite{LX:2015,Nesterov:2012,RT:2014} can be interpreted within the framework of the randomized subspace correction method; see also~\cite[Section~3.3]{JPX:2025} and~\cite[Section~6.4]{Park:2020} for related discussions.

We begin by presenting the standard setting for block coordinate descent methods~\cite{LX:2015,RT:2014}.  
Assume that the space \( V \) is given by the direct sum of subspaces \( V_j \), \( j \in [J] \):
\begin{equation*}
    V = \bigoplus_{j=1}^J V_j,
\end{equation*}
so that each \( v \in V \) can be represented in block form as~\cite[Proposition~1]{RT:2016}
\begin{equation}
\label{block}
    v = (v_1, v_2, \dots, v_J),
    \quad v_j \in V_j.
\end{equation}

In the composite optimization problem~\eqref{model}, we assume that \( F \) is block smooth.
That is, there exist positive constants \( L_j \), \( j \in [J] \), such that
\begin{equation}
\label{F_RBCD}
    F(v + w_j) \leq F(v) + ( F'(v), w_j ) + \frac{L_j}{2} \| w_j \|^2,
    \quad v \in V,\ w_j \in V_j.
\end{equation}
In addition, we assume that the functional \( G \) is block separable, meaning it admits the decomposition
\begin{equation}
\label{G_RBCD}
    G(v) = \sum_{j=1}^J G^j(v_j),
    \quad v \in V,
\end{equation}
for some functionals \( G^j \colon V_j \to \overline{\mathbb{R}} \).

In the local problem~\eqref{local_inexact}, for each \( j \in [J] \), we set
\begin{equation}
\label{local_RBCD}
\begin{aligned}
F_j(w_j; v) &= F(v) + ( F'(v), w_j ) + \frac{L_j}{2} \| w_j \|^2, \\
G_j(w_j; v) &= G ( v + w_j),
\end{aligned}
\quad w_j \in V_j,\ v \in V.
\end{equation}
Then, it is readily seen that the local problem~\eqref{local_inexact} computes the increment of a single proximal descent step with step size \( 1 / L_j \) applied to the \( j \)th block coordinate; see \cref{Ex:local_inexact}.
Consequently, \cref{Alg:RSC} reduces to the randomized block coordinate descent method considered in~\cite{LX:2015,RT:2014}.

Next, we demonstrate that our convergence theory recovers the results established in~\cite{RT:2014} for the randomized block coordinate descent method.
It is straightforward to verify that \cref{Ass:local} holds with $\omega = 1$, based on~\eqref{F_RBCD},~\eqref{G_RBCD}, and~\eqref{local_RBCD}.
Moreover, using~\eqref{G_RBCD}, we obtain (cf.~\cite[Section~6.4]{Park:2020})
\begin{equation*}
    \sum_{j=1}^J G_j(w_j; v) = G(v + w) + (J - 1) G(v),
    \quad v,w \in V.
\end{equation*}
From~\eqref{F_RBCD} and~\eqref{local_RBCD}, we have
\begin{equation*}
    \sum_{j=1}^J d_j(w_j; v) \leq \frac{1}{2} \| w \|_L^2,
    \quad v, w \in V,
\end{equation*}
where the norm \( \| \cdot \|_L \) is defined by
\begin{equation*}
    \| v \|_L = \left( \sum_{j=1}^J L_j \| v_j \|^2 \right)^{\frac{1}{2}},
    \quad v \in V.
\end{equation*}
That is, \cref{Ass:stable} holds with the \( \| \cdot \|_L \)-norm, \( q = 2 \), and \( C_K = 1 \) for any bounded convex subset \( K \subset V \).
Consequently, by \cref{Thm:conv}, the randomized block coordinate descent method satisfies
\begin{equation*}
    \mathbb{E}[E(u^{(n)})] - E(u) \leq \frac{2 J R_0^2}{n + 2 J R_0^2 / ( E(u^{(0)}) - E(u) )},
    \quad n \geq 0,
\end{equation*}
provided that \( E(u^{(0)}) - E(u) \) is sufficiently small.
Here, \( R_0 \) is defined by~\eqref{K_0} with respect to the \( \| \cdot \|_L \)-norm.
This result is consistent with~\cite[Theorem~5]{RT:2014} and~\cite[Equation~(16)]{LX:2015}.

Now suppose further that \( F \) and \( G \) are \( \mu_F \)- and \( \mu_G \)-strongly convex with respect to the \( \| \cdot \|_L \)-norm, respectively, for some \( \mu_F, \mu_G \geq 0 \) with \( \mu_F + \mu_G > 0 \).
In this case, \cref{Thm:conv_strong} implies the following linear convergence rate:
\begin{equation*}
    \mathbb{E}[E(u^{(n)})] - E(u) \leq \left(1 - \frac{1}{J} \frac{\mu_F + \mu_G}{1 + \mu_G}\right)^{n} (E(u^{(0)}) - E(u)),
    \quad n \geq 0,
\end{equation*}
which agrees with~\cite[Theorem~7]{RT:2014}.

\subsection{Operator splitting methods}
\label{Subsec:OS}
By utilizing the duality between subspace correction and operator splitting methods developed in~\cite{JPX:2025}, we can derive a randomized operator splitting method from the randomized subspace correction method.

As a model problem for operator splitting methods, we consider the following optimization problem involving the sum of multiple convex functionals:
\begin{equation}
\label{model_OS}
    \min_{v \in V} \left\{ F(v) + \sum_{j=1}^J G_j (B_j v) \right\},
\end{equation}
where \( V \) and each \( W_j \), \( j \in [J] \), are Hilbert spaces, \( B_j \colon V \to W_j \) are continuous linear operators, and \( F \colon V \to \mathbb{R} \) and \( G_j \colon W_j \to \overline{\mathbb{R}} \) are proper, convex, and lower semicontinuous functionals.
For simplicity, we assume that \( F \) is strongly convex and smooth so that \( F' \) is invertible~\cite{BC:2011}.

A randomized Peaceman--Rachford-type splitting method for solving~\eqref{model_OS} is presented in \cref{Alg:OS}.
At each iteration of \cref{Alg:OS}, an index \( j \in [J] \) is selected at random, and an optimization problem involving only the functional \( G_j \) is solved.
Recall that $d$ denotes the Bregman divergence associated with $F$.

\begin{algorithm}
\caption{Randomized Peaceman--Rachford splitting algorithm for~\eqref{model_OS}}
\begin{algorithmic}[]
\label{Alg:OS}
\STATE Choose $u^{(0)} \in V$ and $v_1^{(0)},\dots,v_J^{(0)} \in V$.
\FOR{$n=0,1,2,\dots$}
    \STATE Sample $j \in [J]$ from the uniform distribution on $[J]$.
    \STATE $\displaystyle
        u^{(n+1)} = \operatornamewithlimits{\arg\min}_{v \in V} \left\{ d (v - u^{(n)}; u^{(n)}) + (v_j^{(n)}, v) + G_j (B_j v) \right\}
        $
    \STATE $\displaystyle
        v_i^{(n+1)} = \begin{cases}
            v_i^{(n)} + F' (u^{(n+1)}) - F'(u^{(n)}), & \quad \text{ if } i = j, \\
            v_i^{(n)}, & \quad \text{ if } i \neq j,
        \end{cases}
        \quad i \in [J]
        $
\ENDFOR
\end{algorithmic}
\end{algorithm}

Similar to~\cite[Theorem~5.4]{JPX:2025}, \cref{Thm:OS} shows that \cref{Alg:OS} is a \textit{dualization} of the randomized subspace correction method applied to the following dual problem:
\begin{equation}
\label{model_OS_dual}
\min_{(p_j)_{j=1}^J \in \bigoplus_{j=1}^J W_j} \left\{ F^* \left( - \sum_{j=1}^J B_j^* p_j \right) + \sum_{j=1}^J G_j^* (p_j) \right\},
\end{equation}
where \( F^* \colon V \to \mathbb{R} \) and \( G_j^* \colon W_j \to \overline{\mathbb{R}} \) denote the Legendre--Fenchel conjugates of \( F \) and \( G_j \), respectively, and \( B_j^* \) denotes the adjoint of \( B_j \).
Note that~\eqref{model_OS_dual} is an instance of~\eqref{model} in which the nonsmooth part is block separable~(cf.~\eqref{G_RBCD}).

\begin{theorem}
\label{Thm:OS}
Let $\{ u^{(n)} \}$, $\{ (v_j^{(n)})_{j=1}^J \}$, and $\{ (p_j^{(n)} )_{j=1}^J \}$ be the sequences generated by the randomized Peaceman--Rachford splitting algorithm for solving~\eqref{model_OS}~(\cref{Alg:OS}) and the randomized subspace correction method with exact local problems for solving~\eqref{model_OS_dual}.
If
\begin{equation*}
u^{(0)} = (F^*)' \left( - \sum_{j=1}^J B_j^* p_j^{(0)} \right), \quad
v_j^{(0)} = - B_j^* p_j^{(0)},
\quad j \in [J],
\end{equation*}
then we have
\begin{equation*}
u^{(n)} \stackrel{\text{a.s.}}{=} (F^*)' \left( - \sum_{j=1}^J B_j^* p_j^{(n)} \right), \quad
v_j^{(n)} \stackrel{\text{a.s.}}{=} - B_j^* p_j^{(n)},
\quad j \in [J],\ n \geq 1.
\end{equation*}
\end{theorem}
\begin{proof}
See~\cite[Theorem~5.4]{JPX:2025}.
\end{proof}
\section{Applications}
\label{Sec:Applications}
In this section, we present applications of the randomized subspace correction method to a range of problems arising in diverse areas of science and engineering.
See \cref{Table:applications} for a concise summary.

\begin{table}
    \centering
    \resizebox{\textwidth}{!}{
    \begin{tabular}{c|cc}
        \textbf{Applications} & \textbf{Stable decompositions} & \textbf{References} \\
        \hline
        linear problems & any & \makecell{randomized methods~\cite{GO:2012,HXZ:2019} \\ examples~\cite{GO:2024,Park:2024b,TW:2005}} \\
        \hline
        nonlinear PDEs  & \makecell{domain decomposition \\ multigrid} & \makecell{$s$-Laplacian~\cite{LP:2025a,TX:2002} \\ other examples~\cite{CHW:2020a,Park:2024a}}\\
        \hline
        variational inequalities & \makecell{domain decomposition \\ multigrid} & \makecell{second-order~\cite{BTW:2003,Tai:2003} \\ other examples~\cite{BK:2012,Carstensen:1997,Park:2024c}} \\
        \hline
        total variation minimization & domain decomposition & \makecell{dual form~\cite{CTWY:2015,HL:2015,LP:2019,Park:2021a} \\ primal form~\cite{LG:2019,LN:2017}} \\
        \hline
        multinomial logistic regression & nonoverlapping blocks & \makecell{dual form~\cite{SZ:2013,SZ:2014} \\ primal form~\cite{JPX:2025}} 
    \end{tabular}
    }
    \caption{Applications of the randomized subspace correction method discussed in this paper. Possible stable decompositions and relevant references are provided.}
    \label{Table:applications}
\end{table}

Following standard convention in the literature~(see, e.g.,~\cite{Xu:1992}), we write $x \lesssim y$, or equivalently $y \gtrsim x$, to mean that there exists a constant $C>0$, independent of the relevant parameters, such that $x \leq C y$. 
Furthermore, we write $x \eqsim y$ when both $x \lesssim y$ and $x \gtrsim y$ hold.

\subsection{Linear problems}
A fundamental class of problems covered by the randomized subspace correction method is the class of linear problems~\cite{GO:2012,HXZ:2019}, as discussed in \cref{Subsec:linear}. 
Given the extensive literature on subspace correction methods for linear systems, particularly those arising from the numerical discretization of elliptic PDEs, we omit detailed discussion for brevity. 
We refer the reader to~\cite{GO:2024,Park:2024b,TW:2005} and the references therein.

\subsection{Nonlinear partial differential equations}
The randomized subspace correction method is also applicable to nonlinear PDEs that admit convex variational formulations.
As an example, we briefly present the \( s \)-Laplacian problem, which was also considered in~\cite{LP:2025a,TX:2002}.
For other examples of nonlinear PDEs, one may refer to, e.g.,~\cite{CHW:2020a,Park:2024a}.

We consider the following nonlinear problem:
\begin{equation*}
\begin{split}
- \operatorname{div} ( | \nabla u |^{s-2} \nabla u ) = f & \quad \text{ in } \Omega, \\
u = 0 & \quad \text{ on } \partial \Omega,
\end{split}
\end{equation*}
where $\Omega \subset \mathbb{R}^2$ is a bounded polygonal domain, $s  \in (1, \infty)$ with $s \neq 2$, and $f \in W^{-1, s^*} (\Omega)$ with $1/s + 1/s^* = 1$.
It is well-known that the above problem admits a weak formulation given by the following convex optimization problem:
\begin{equation}
\label{sLap}
\min_{v \in W_0^{1,s} (\Omega) } \left\{ \frac{1}{s} \int_{\Omega} | \nabla v |^s \,dx - \langle f, v \rangle \right\}.
\end{equation}

To numerically solve~\eqref{sLap}, we employ a finite element discretization.
Let \( \mathcal{T}_h \) be a quasi-uniform triangulation of \( \Omega \), where \( h \) denotes the characteristic element diameter.
We denote by \( S_h(\Omega) \) the lowest-order Lagrangian finite element space defined on \( \mathcal{T}_h \), incorporating the homogeneous essential boundary condition.
Then, the finite element approximation of~\eqref{sLap} is given by
\begin{equation}
\label{sLap_FEM}
\min_{v \in S_h(\Omega)} \left\{ \frac{1}{s} \int_{\Omega} |\nabla v|^s \,dx - \langle f, v \rangle \right\}.
\end{equation}
We observe that~\eqref{sLap_FEM} is an instance of the abstract problem~\eqref{model}.
Namely,~\eqref{sLap_FEM} corresponds to~\eqref{model} with
\begin{equation*}
V = S_h(\Omega), \quad
F(v) = \frac{1}{s} \int_{\Omega} |\nabla v|^s \,dx - \langle f, v \rangle, \quad
G(v) = 0.
\end{equation*}

For the space decomposition~\eqref{space_decomposition}, two-level overlapping Schwarz and multigrid decompositions with exact local solvers were studied in~\cite{LP:2025a,Park:2020} and~\cite{TX:2002}, respectively.
Here, we analyze the two-level overlapping Schwarz method for solving~\eqref{sLap_FEM}~\cite{LP:2025a,Park:2020,TX:2002} within the proposed framework.
Let $\mathcal{T}_H$ be a quasi-uniform triangulation of $\Omega$ such that $\mathcal{T}_h$ is a refinement of $\mathcal{T}_H$, where $H$ denotes the characteristic element diameter of $\mathcal{T}_H$.
Let $\{ \Omega_i \}_{i=1}^N$ be a quasi-uniform overlapping domain decomposition of $\Omega$, where each subdomain $\Omega_i$ is a union of $\mathcal{T}_h$-elements with diameter of order $H$.
The overlap width is denoted by $\delta$.
We note that, for any bounded and convex subset $K \subset V$, the following estimate holds~\cite[equations~(6.6) and~(6.7)]{Park:2020}:
\begin{equation}
\label{sLap_Bregman}
\frac{\alpha_{K}}{p} \| u - v \|_{W^{1,s} (\Omega)}^p
\leq F(u) - F(v) - \langle F'(v), u - v \rangle
\leq \frac{\beta_{K}}{q} \|u - v \|_{W^{1,s} (\Omega)}^q,
\quad u,v \in K,
\end{equation}
where
\begin{equation*}
p = \max \{ s, 2 \}, \quad
q = \min \{ s, 2 \},
\end{equation*}
and $\alpha_K$, $\beta_K$ are positive constants depending on $K$.

In the two-level overlapping Schwarz method, we set
\begin{equation}
\label{two_level}
V_{N+1} = S_H (\Omega), \quad
V_i = S_h (\Omega_i), \quad i \in [N],
\end{equation}
where $S_H (\Omega)$ and $S_h (\Omega_i)$ are defined analogously to $S_h (\Omega)$.
For the two-level space decomposition with $J=N+1$,
\begin{equation}
\label{two_level_space_decomposition}
V = \sum_{j=1}^J V_j,
\end{equation}
the following stable decomposition result for~\eqref{two_level} with respect to the $\| \cdot \|_{W^{1,s}} (\Omega)$-norm is available in~\cite{LP:2025a,Park:2020,TX:2002}: for $w \in V$, there exist $w_j \in V_j$, $1 \leq j \leq J$, such that $w = \sum_{j=1}^J w_j$ and
\begin{equation}
\label{sLap_stable}
\sum_{j=1}^J \| w_j \|_{W^{1,s} (\Omega)}^s \lesssim \left( 1 + \left( \frac{H}{\delta} \right)^{s-1} \right) \| w \|_{W^{1,s} (\Omega)}^s.
\end{equation}
By~\eqref{sLap_Bregman} and~\eqref{sLap_stable}, \cref{Ass:stable} holds with
\begin{equation*}
C_{K_0} \eqsim 1 + \left( \frac{H}{\delta} \right)^{\frac{q(s-1)}{s}}.
\end{equation*}
Assuming that all local problems are solved exactly, i.e., using~\eqref{local_exact}, \cref{Ass:local} holds trivially.
Moreover, by~\eqref{sLap_Bregman}, \cref{Ass:sharp} holds with $\mu \eqsim 1$.
Invoking \cref{Thm:conv_sharp}, we obtain
\begin{equation*}
\mathbb{E}[E(u^{(n)})] - E(u)
\lesssim \frac{1 + (H/\delta)^{\frac{q(s-1)}{s}}}{(n/J)^{\frac{p(q-1)}{p-q}}}
\end{equation*}
That is, the two-level randomized Schwarz method for solving~\eqref{sLap_FEM} converges sublinearly with rate $\mathcal{O} (n^{- \frac{p(q-1)}{p-q}} )$.

\subsection{Variational inequalities}
Another important class of problems is variational inequalities, which find applications in computational mechanics and optimal control.
As an illustrative example, we consider a second-order problem~\cite{BTW:2003,Tai:2003}; see~\cite{BK:2012,Carstensen:1997,Park:2024c} for further examples.

We consider the following variational inequality: find \( u \in K \) such that
\begin{equation*}
\int_{\Omega} \nabla u \cdot \nabla v \,dx - \langle f, v \rangle \geq 0,
\quad v \in K,
\end{equation*}
where $\Omega \subset \mathbb{R}^2$ is a bounded polygonal domain, \( f \in H^{-1}(\Omega) \), and \( K \) is a subset of \( H_0^1(\Omega) \) representing a pointwise inequality constraint:
\begin{equation*}
K = \{ v \in H_0^1(\Omega) : v \geq g \text{ a.e. in } \Omega \},
\end{equation*}
for some \( g \in C(\Omega) \).
This problem admits the equivalent optimization formulation
\begin{equation*}
\min_{v \in K} \left\{ \frac{1}{2} \int_{\Omega} |\nabla v|^2 \,dx - \langle f, v \rangle \right\}.
\end{equation*}
Using a finite element discretization defined on \( S_h(\Omega) \), we obtain the discrete problem
\begin{equation}
\label{VI_FEM}
\min_{v \in S_h (\Omega)} \left\{ \frac{1}{2} \int_{\Omega} |\nabla v|^2 \,dx - \langle f, v \rangle + \chi_{K_h} (v) \right\},
\end{equation}
where \( K_h = K \cap S_h(\Omega) \), and \( \chi_{K_h} \) denotes the indicator functional of \( K_h \), defined as \( \chi_{K_h}(v) = 0 \) if \( v \in K_h \) and \( \chi_{K_h}(v) = \infty \) otherwise.
We observe that~\eqref{VI_FEM} is an instance of the general form~\eqref{model}.
Specifically, we obtain~\eqref{VI_FEM} if we set
\begin{equation*}
V = S_h(\Omega), \quad
F(v) = \frac{1}{2} \int_{\Omega} |\nabla v|^2 \,dx - \langle f, v \rangle, \quad
G(v) = \chi_{K_h}(v).
\end{equation*}

Two-level overlapping domain decomposition and multigrid decomposition with exact local solvers were studied in~\cite{Badea:2006,BTW:2003}.
In addition, constraint decomposition methods, based on localized constraints and interpretable as instances of inexact local solvers within our framework~\cite[Section~6.4]{Park:2020}, were considered in~\cite{BF:2024,Tai:2003}.
Here, we analyze the two-level overlapping Schwarz method for solving~\eqref{VI_FEM}, using the same subspaces as in~\eqref{two_level} and the space decomposition~\eqref{two_level_space_decomposition}.
We first observe that, since $F$ is quadratic, by the Poincar\'{e} inequality we readily have
\begin{equation}
\label{VI_Bregman}
    F(u) - F(v) - \langle F'(v), u-v \rangle \eqsim \| u - v \|_{H^1 (\Omega)}^2, \quad u,v \in V.
\end{equation}
Moreover,~\cite[Theorem~4.3]{Park:2024a}~(see also~\cite{BTW:2003,Tai:2003}) proves the following stable decomposition result for~\eqref{two_level} with respect to the $\| \cdot \|_{H^1 (\Omega)}$-norm: for $v,w \in V$, there exists $w_j \in V_j$, $1 \leq j \leq J$, such that $w = \sum_{j=1}^J w_j$ and
\begin{subequations}
\label{VI_stable}
\begin{align}
\sum_{j=1}^J \| w_j \|_{H^1 (\Omega)}^2 \lesssim \left(1 + \log \frac{H}{h} \right) \left( 1 + \frac{H}{\delta} \right) \| w \|_{H^1 (\Omega)}^2, \\
\sum_{j=1}^J G(v + w_j) \leq G(v + w) + J G(v).
\end{align}
\end{subequations}
By~\eqref{VI_Bregman} and~\eqref{VI_stable}, \cref{Ass:stable} holds with $q=2$ and
\begin{equation*}
    C_V \eqsim \left( 1 + \log \frac{H}{h} \right) \left(1 + \frac{H}{\delta} \right).
\end{equation*}
Assuming that all local problems are solved exactly, i.e., using~\eqref{local_exact}, \cref{Ass:local} holds trivially.
Moreover, by~\eqref{VI_Bregman}, \cref{Ass:strong} holds with $\mu \eqsim 1$.
Invoking \cref{Thm:conv_strong}, we obtain
\begin{equation*}
\mathbb{E} [E (u^{(n)})] - E(u) \leq \left( 1- \frac{1}{J C_V} \right)^n ( E (u^{(0)}) - E(u) ) .
\end{equation*}
That is, the two-level randomized Schwarz method for solving~\eqref{VI_FEM} converges linearly.

\subsection{Total variation minimization}
Total variation minimization is a fundamental problem in mathematical imaging; see, e.g.,~\cite{CP:2016}.
Given a bounded polygonal domain \( \Omega \subset \mathbb{R}^2 \), we consider the variational problem
\begin{equation}
\label{TV}
\min_{v \in BV(\Omega)} \left\{ \frac{1}{2} \int_{\Omega} (v - f)^2 \,dx + TV(v) \right\},
\end{equation}
where \( TV(v) \) denotes the total variation of \( v \), \( BV(\Omega) \) is the space of functions of bounded variation, and \( f \in L^2(\Omega) \).

Designing subspace correction methods for solving~\eqref{TV} is particularly challenging due to the nonseparable structure of the total variation term~\cite{Park:2021a}.
Indeed, it was shown in~\cite{LN:2017} that standard domain decomposition methods generally fail to satisfy the stable decomposition condition stated in \cref{Ass:stable}.

One viable approach is to instead consider the dual formulation of~\eqref{TV}~\cite{KH:2004}, which reads:
\begin{equation}
\label{TV_dual}
\min_{\mathbf{p} \in H_0 (\operatorname{div}; \Omega)} \frac{1}{2} \int_{\Omega} (\operatorname{div} \mathbf{p} + f)^2 \,dx
\quad \text{subject to} \quad
\| \mathbf{p} \|_{L^{\infty} (\Omega)} \leq 1.
\end{equation}
By employing the Raviart--Thomas finite element discretization introduced in~\cite{LPP:2019}, we obtain the discrete problem
\begin{equation}
\label{TV_dual_FEM}
\min_{\mathbf{p} \in \mathbf{S}_h (\Omega)} \left\{ \frac{1}{2} \int_{\Omega} (\operatorname{div} \mathbf{p} + f)^2 \,dx + \chi_{\mathbf{K}_h} (\mathbf{p}) \right\},
\end{equation}
where \( \mathbf{S}_h (\Omega) \) denotes the lowest-order Raviart--Thomas finite element space on the mesh \( \mathcal{T}_h \), and \( \mathbf{K}_h \subset \mathbf{S}_h (\Omega) \) is a convex set encoding the constraint \( \| \mathbf{p} \|_{L^{\infty} (\Omega)} \leq 1 \).
The discrete problem~\eqref{TV_dual_FEM} fits the abstract formulation~\eqref{model}, with the following identifications:
\begin{equation*}
V = \mathbf{S}_h (\Omega), \quad
F(\mathbf{p}) = \frac{1}{2} \int_{\Omega} (\operatorname{div} \mathbf{p} + f)^2 \,dx, \quad
G(\mathbf{p}) = \chi_{\mathbf{K}_h} (\mathbf{p}).
\end{equation*}

Schwarz-type domain decomposition methods for the dual problem~\eqref{TV_dual_FEM} were analyzed in~\cite{Park:2021a}, and related constraint decomposition techniques were introduced in~\cite{CTWY:2015}; see also~\cite{HL:2015,LP:2019} for domain decomposition methods for finite difference discretizations of~\eqref{TV_dual}.
Here, we analyze the one-level overlapping Schwarz method for solving~\eqref{TV_dual_FEM}.
In this method, we assume that the domain $\Omega$ is decomposed into overlapping subdomains $\{ \Omega_j \}_{j=1}^J$ as before.
We set
\begin{equation*}
    V_j = \mathbf{S}_h (\Omega_j),
    \quad j \in [J],
\end{equation*}
where $\mathbf{S}_h (\Omega_j)$ is defined analogously to $\mathbf{S}_h (\Omega)$.
Then we obtain the one-level space decomposition~\eqref{space_decomposition}.

We observe that $F$ satisfies
\begin{equation}
\label{TV_Bregman}
    F (\mathbf{p}) - F(\mathbf{q}) - \langle F'(\mathbf{q}), \mathbf{p} - \mathbf{q} \rangle = \frac{1}{2} \| \operatorname{div} ( \mathbf{p} - \mathbf{q} ) \|_{L^2 (\Omega)}^2,
    \quad \mathbf{p},\ \mathbf{q} \in V.
\end{equation}
It was shown in~\cite[Lemma~4.2]{Park:2021a} that the following one-level stable decomposition with respect to the $\| \cdot \|_{H (\operatorname{div}; \Omega)}$-norm holds:
for $\mathbf{q}, \mathbf{r} \in V$, there exist $\mathbf{r}_j \in V_j$, $j \in [J]$, such that $\mathbf{r} = \sum_{j=1}^J \mathbf{r}_j$ and
\begin{subequations}
\label{TV_stable}
\begin{align}
\sum_{j=1}^J \| \operatorname{div} \mathbf{r}_j \|_{L^2 (\Omega)}^2 \lesssim \frac{1}{\delta^2} \| \mathbf{r} \|_{H (\operatorname{div}; \Omega)}^2, \\
\sum_{j=1}^J G(\mathbf{q} + \mathbf{r}_j) \leq G(\mathbf{q} + \mathbf{r}) + (J-1) G(\mathbf{q}).
\end{align}
\end{subequations}
By~\eqref{TV_Bregman} and~\eqref{TV_stable}, \cref{Ass:stable} holds with $q=2$ and
\begin{equation*}
    C_V \eqsim \frac{1}{\delta^2}.
\end{equation*}
Assuming that all local problems are solved exactly, \cref{Ass:local} holds.
Invoking \cref{Thm:conv}, we obtain
\begin{equation*}
\mathbb{E}[E(\mathbf{p}^{(n)})] - E(\mathbf{p})
\lesssim \frac{\delta^{-2}}{n/J},
\end{equation*}
where $\mathbf{p} \in V$ denotes a solution of~\eqref{TV_dual_FEM}.
That is, the one-level randomized Schwarz method for solving~\eqref{TV_dual_FEM} converges sublinearly with rate $\mathcal{O}(n^{-1})$.

Meanwhile, to tackle the primal problem~\eqref{TV} directly, one may employ the operator splitting approach described in \cref{Subsec:OS} to design domain decomposition methods.
In this context, it was shown in~\cite[Section~5.5.1]{JPX:2025} that the resulting methods coincide with those proposed in~\cite{LN:2017}; see also~\cite{LG:2019}.

\subsection{Multinomial logistic regression}
Multinomial logistic regression is a fundamental tool for classification problems in machine learning.  
Given a labeled dataset \( \{ (x_{i}, y_{i} )\}_{i=1}^{N} \subset \mathbb{R}^d \times [k] \), where each \( x_{i} \in \mathbb{R}^d \) represents a data point and \( y_{i} \in [k] \) its corresponding class label, the regression model is formulated as follows:
\begin{equation}
\label{logistic_regression}
    \min_{\theta \in \mathbb{R}^{(d+1)k}}
    \left\{ \frac{1}{N} \left( \sum_{i=1}^{N} \operatorname{LSE}_k (X_{i}^T \theta) - \hat{x}^T \theta \right)  + \frac{\alpha}{2} \| \theta \|^2 \right\},
\end{equation}
where \( \theta = [w_1^T, b_1, \dots, w_k^T, b_k]^T \in \mathbb{R}^{(d+1)k} \) is the parameter vector, \( \operatorname{LSE}_k \) denotes the log-sum-exp function over \( k \) classes, \( X_{i} \in \mathbb{R}^{(d+1)k \times k} \) is defined as \( X_{i} = I_k \otimes [x_{i}^T, 1]^T \), and $\alpha$ is a positive regularization hyperparameter.
The vector \( \hat{x} \in \mathbb{R}^{(d+1)k} \) is a constant vector depending only on the dataset; see~\cite[Example~2.10]{JPX:2025} for details.
Note that~\eqref{logistic_regression} is an instance of~\eqref{model_OS} with $J = N$.

In big data settings, where the number of data points \( N \) is very large, operator splitting methods such as stochastic gradient descent~\cite{Bertsekas:2011}, which process individual data points or mini-batches, are widely used to solve~\eqref{logistic_regression}.
The randomized Peaceman--Rachford splitting algorithm presented in \cref{Alg:OS} is one such method.

Thanks to \cref{Thm:OS}, the convergence analysis of \cref{Alg:OS} for solving~\eqref{logistic_regression} reduces to analyzing the randomized subspace correction method for the following dual problem~\cite{SZ:2013,SZ:2014}:
\begin{equation}
\label{logistic_regression_dual}
\min_{(p_{i})_{i=1}^{N} \in (\mathbb{R}^{k})^N} \left\{ \frac{1}{2 N \alpha} \left\| \sum_{i=1}^{N} X_{i} p_{i} - \hat{x} \right\|^2 + \sum_{i=1}^{N} \operatorname{LSE}_k^* (p_{i}) \right\}.
\end{equation}
More precisely, we consider the randomized subspace correction method for solving~\eqref{logistic_regression_dual} with a nonoverlapping $N$-block decomposition, i.e., in~\eqref{space_decomposition} we set
\begin{equation*}
    J = N, \quad
    V = (\mathbb{R}^{k})^N, \quad
    V_i = \mathbb{R}^k, \quad i \in [N].
\end{equation*}
Since the decomposition is nonoverlapping, each $p \in V$ admits a unique decomposition $p = (p_i)_{i=1}^N$ with respect to the direct sum $\bigoplus_{i=1}^N V_i$.
Moreover,~\eqref{logistic_regression_dual} is an instance of~\eqref{model} with
\begin{equation*}
    F(p) = \frac{1}{2 N \alpha} \left\| \sum_{i=1}^{N} X_{i} p_{i} - \hat{x} \right\|^2, \quad
    G(p) = \sum_{i=1}^{N} \operatorname{LSE}_k^* (p_{i}).
\end{equation*}
Since the term \( \sum_{i=1}^{N} \operatorname{LSE}_k^* (p_{i}) \) is block separable, an argument analogous to that in \cref{Subsec:BCD} shows that \cref{Ass:stable} holds with \( q = 2 \).
Consequently, \cref{Thm:conv} applies to establish convergence of the randomized subspace correction method for solving~\eqref{logistic_regression_dual}.

\section{Numerical results}
\label{Sec:Numerical}
In this section, we present numerical results for some applications of the randomized subspace correction method.
All algorithms were implemented in MATLAB~R2024b and run on a desktop equipped with an AMD Ryzen~5 5600X CPU~(3.7\,GHz, 6C), 40\,GB RAM, and Windows~10 Pro.

\subsection{\texorpdfstring{$s$}{s}-Laplacian problem}
As the first example, we consider two-level overlapping Schwarz methods for the $s$-Laplacian problem~\eqref{sLap}.
The additive and randomized Schwarz methods correspond to the parallel and randomized subspace correction methods in \cref{Alg:PSC,Alg:RSC}, respectively.

In our numerical experiments, we adopt the setting in~\cite{LP:2025a}.
We set $\Omega = (0,1)^2 \subset \mathbb{R}^2$, $s \in \{1.5, 5\}$, and $f = 1$.
For the finite element discretization~\eqref{sLap_FEM} and the two-level space decomposition~\eqref{two_level_space_decomposition}, we partition $\Omega$ into $2 \times 1/H \times 1/H$ uniform triangles to obtain a coarse triangulation $\mathcal{T}_H$, and then refine it to obtain a fine triangulation $\mathcal{T}_h$ with $2 \times 1/h \times 1/h$ triangles.
Each nonoverlapping subdomain $D_i$, $i \in [N]$ with $N = 1/H \times 1/H$, is defined as a rectangular region consisting of two coarse triangles sharing a diagonal edge.
The corresponding overlapping subdomain $D_i'$ is obtained by enlarging $D_i$ with layers of fine triangles of width $\delta$.
The decomposition $\{ D_i' \}_{i=1}^N$ can be colored with four colors when $\delta$ is sufficiently small; see~\cite[Figure~1]{TX:2002}.
Let $\Omega_j$, $j \in [N_c]$ with $N_c = 4$, denote the union of all $D_i'$ with the $j$th color.
Defining $V_j$, $j \in [J]$ with $J = N_c + 1 = 5$, as in~\eqref{two_level}, we obtain the two-level space decomposition~\eqref{two_level_space_decomposition}.
This coloring-based space decomposition enables parallel computation in \cref{Alg:RSC}: all subdomain problems in the same color can be solved in parallel, while in \cref{Alg:PSC} all subdomain problems regardless of color can be solved in parallel.
In \cref{Alg:PSC}, we choose the step size $\tau = 1/5$ (cf.~\cite{LP:2025a}).

In both \cref{Alg:PSC,Alg:RSC}, we initialize $u^{(0)}$ randomly.
For solving the local and coarse problems on $V_j$, $1 \le j \le J$, we use the adaptive Newton method in~\cite{Mishchenko:2023} with the stopping criterion
\begin{equation}
\label{stop}
\left| \frac{E_j^{(n)}(w_j^{(m+1)}) - E_j^{(n)}(w_j^{(m)})}{E_j^{(n)}(w_j^{(m+1)})} \right| < 10^{-12}.
\end{equation}
Here $E_j^{(n)}$ denotes the energy functional associated with $V_j$ at the $n$th outer iteration, and $m$ is the number of inner iterations.
To obtain a reference solution $u_h$, we apply the adaptive Newton method~\cite{Mishchenko:2023} to the full problem~\eqref{sLap_FEM} with sufficiently many iterations.

\begin{figure}
\centering
\subfloat[][$s = 1.5$]{ \includegraphics[width=0.3\linewidth]{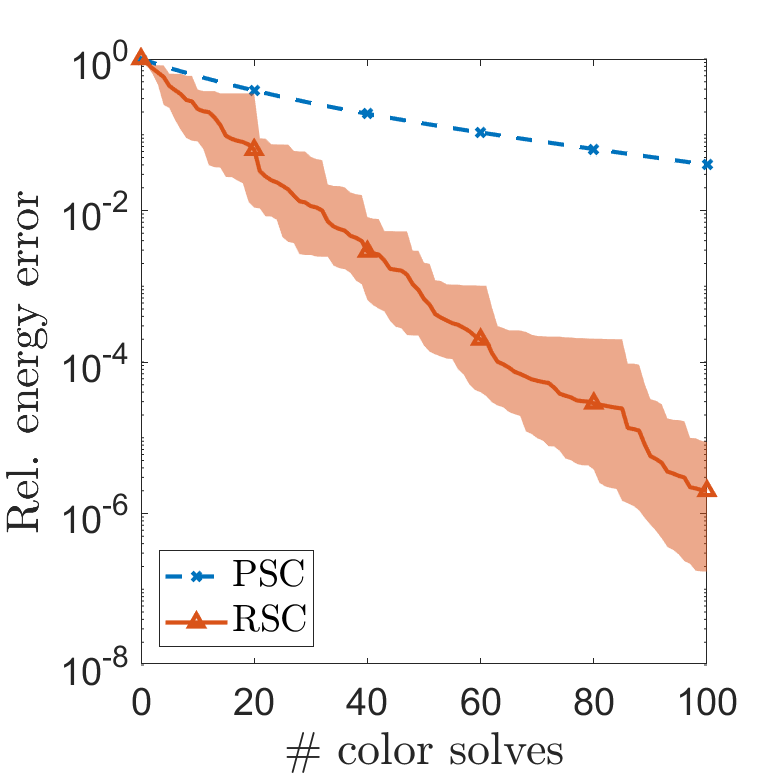} }
\quad
\subfloat[][$s = 5.0$]{ \includegraphics[width=0.3\linewidth]{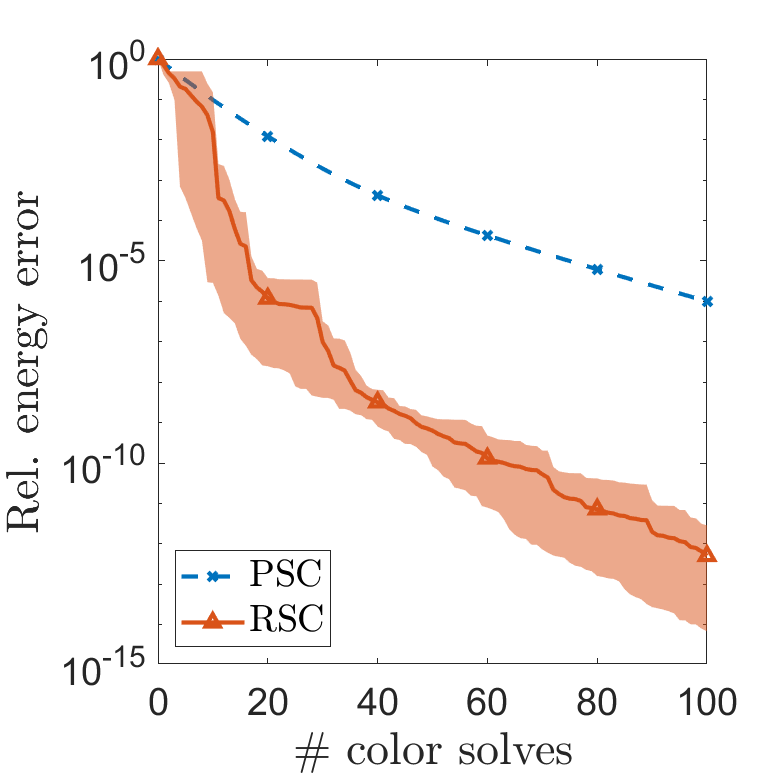} }
\caption{Decay of the relative energy error~\eqref{rel_energy_error} in Algorithms~\ref{Alg:PSC}~(PSC) and~\ref{Alg:RSC}~(RSC) for the $s$-Laplacian problem~\eqref{sLap} based on a two-level domain decomposition~($h = 2^{-6}$, $H = 2^{-2}$, $\delta = h$).
The RSC curve is obtained by averaging over 10 independent experiments, and the shaded region indicates the range from the minimum to the maximum values across all experiments.}
\label{Fig:sLap}
\end{figure}

In \cref{Fig:sLap}, we depict convergence curves of the relative energy error
\begin{equation}
\label{rel_energy_error}
\frac{E(u^{(n)}) - E(u)}{E(u^{(0)}) - E(u)}
\end{equation}
with respect to the number of color solves generated by \cref{Alg:PSC,Alg:RSC} based on the two-level space decomposition~\eqref{two_level_space_decomposition}.
Note that each outer iteration of \cref{Alg:PSC} is counted as $J+1 = 5$ color solves.
We observe that \cref{Alg:RSC} converges faster than \cref{Alg:PSC}, reflecting the well-known fact that successive methods such as Gauss--Seidel converge faster than their parallel counterparts such as Jacobi.
Moreover, \cref{Alg:RSC} is robust with respect to the sampling of subspace indices: even the worst case over all experiments, represented by the upper boundary of the shaded region in \cref{Fig:sLap}, outperforms \cref{Alg:PSC}.

\subsection{\texorpdfstring{$L^1$}{L1}-penalized problem}
As the second example, we consider two-level overlapping Schwarz methods~\cite[Section~5.3]{Park:2024a} for the $L^1$-penalized problem introduced in~\cite{TSFO:2015}:
\begin{equation}
\label{L1}
    \min_{v \in H_0^1 (\Omega)} \int_{\Omega} \left( \frac{1}{2} | \nabla v |^2 - gv + \alpha | v | \right) \,dx,
\end{equation}
where $\alpha > 0$.
In our numerical experiments, we adopt the setting in~\cite{Park:2024a}.
We set $\Omega = (0,1)^2 \subset \mathbb{R}^2$, $g(x,y) = 10^3 x(1-x) \sin(\pi y)$, and $\alpha \in \{10, 30\}$.
For discretization and domain decomposition, we use the same settings as for the $s$-Laplacian problem~\eqref{sLap}.
Note that the corresponding finite element discretization of~\eqref{L1} is an instance of~\eqref{model} with
\begin{equation*}
    V = S_h (\Omega), \quad
    F(v) = \frac{1}{2} \int_{\Omega} | \nabla v |^2 \,dx - \int_{\Omega} g v \,dx, \quad
    G(v) = \alpha \int_{\Omega} | v | \,dx.
\end{equation*}

Since~\eqref{L1} is nonsmooth, the adaptive Newton method introduced above is not applicable.
We employ the fast gradient method with adaptive restart (AFGM)~\cite{OC:2015}.
For solving the local problems, we use AFGM with the stopping criterion~\eqref{stop}.
For the coarse problems, we apply AFGM to the dual formulations of the coarse problems, akin to the approach introduced in~\cite[Proposition~5.1]{Park:2024c}.
A reference solution $u_h$ is obtained by running AFGM on the full problem with sufficiently many iterations.

\begin{figure}
\centering
\subfloat[][$\alpha = 10$]{ \includegraphics[width=0.3\linewidth]{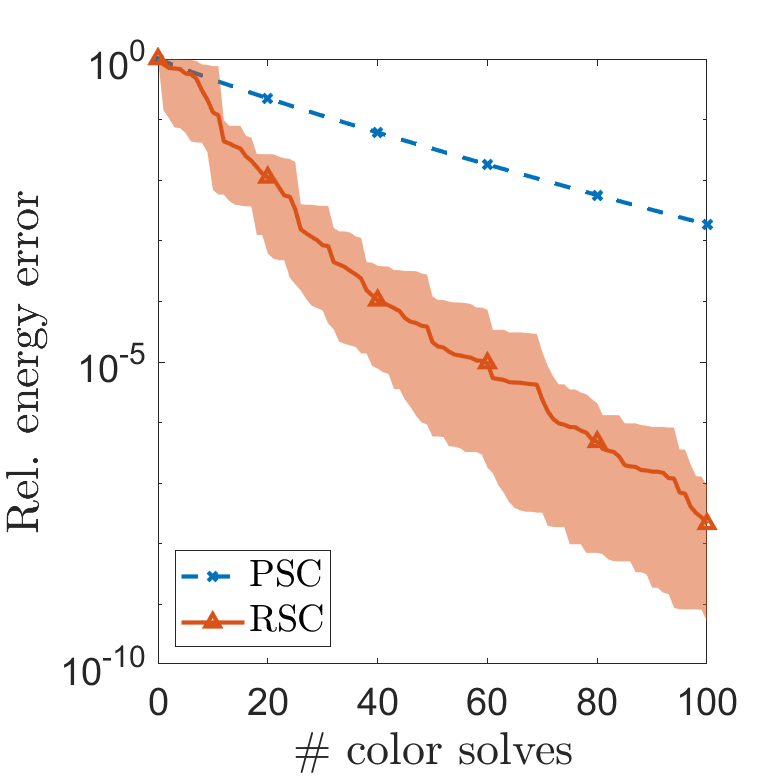} }
\quad
\subfloat[][$\alpha = 30$]{ \includegraphics[width=0.3\linewidth]{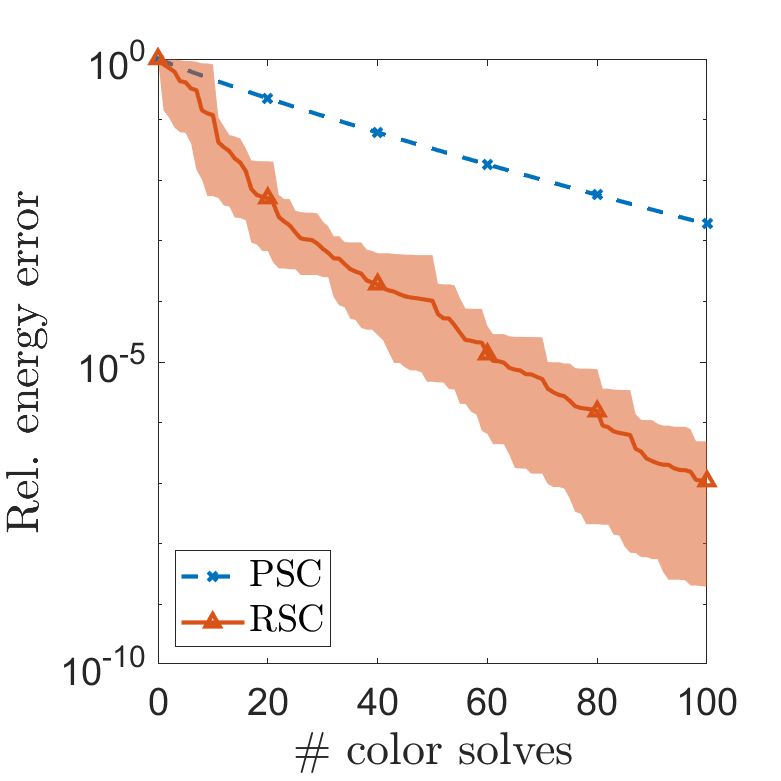} }
\caption{Decay of the relative energy error~\eqref{rel_energy_error} in Algorithms~\ref{Alg:PSC}~(PSC) and~\ref{Alg:RSC}~(RSC) for the $L^1$-penalized problem~\eqref{L1} based on a two-level domain decomposition~($h = 2^{-6}$, $H = 2^{-2}$, $\delta = h$).
The RSC curve is obtained by averaging over 10 independent experiments, and the shaded region indicates the range from the minimum to the maximum values across all experiments.}
\label{Fig:L1}
\end{figure}

In \cref{Fig:L1}, we depict convergence curves of \cref{Alg:PSC,Alg:RSC} based on the two-level space decomposition~\eqref{two_level_space_decomposition}.
A similar discussion applies as in the case of the $s$-Laplacian problem.
Namely, \cref{Alg:RSC} converges faster than \cref{Alg:PSC} with respect to the number of color solves, and is also robust with respect to the sampling of subspace indices.

\section{Conclusion}
\label{Sec:Conclusion}
In this paper, we introduced an abstract framework for randomized subspace correction methods for convex optimization.
We established convergence theorems that are applicable to a broad range of scenarios involving space decomposition, local solvers, and varying levels of problem regularity.
Furthermore, we demonstrated that these theorems unify and extend several relevant recent results.

This work opens several avenues for future research.
One important direction is the development of convergence theory for cyclic successive subspace correction methods. 
This is particularly relevant for multigrid methods~\cite{Badea:2006,TX:2002}, where the hierarchical structure of subspaces plays a critical role.
We note that a sharp convergence analysis of cyclic methods for linear problems can be found in~\cite{Brenner:2013,XZ:2002}.
The block coordinate descent method with $q = 2$ was analyzed in~\cite{BT:2013}; however, the results there are not sharp in the sense that the analysis does not recover the sharp estimate established for the linear case.

Another promising direction is to extend the proposed framework to nonconvex problems.
Recent work~\cite{CN:2023} has shown that randomized block coordinate descent methods are effective for a certain class of nonconvex problems.
Generalizing the randomized subspace correction framework to accommodate such settings remains an interesting and challenging problem.


\bibliographystyle{siamplain}
\bibliography{refs_SC}

\end{document}